\documentclass[11pt]{amsart}
\usepackage{amssymb}
\usepackage{amsmath}
\usepackage{graphicx,psfrag}
\usepackage{fancyhdr}
\usepackage[top=1.4in, bottom=1.1in, left=1.3in, right=1.3in]{geometry}

\usepackage{enumerate}

\title[Arbitrary Orientations of Hamilton Cycles in Digraphs]{Arbitrary Orientations of Hamilton Cycles in Digraphs}

\date{\today}
\author[L. DeBiasio, D. K\"uhn, T. Molla, D. Osthus and A. Taylor]{Louis DeBiasio, Daniela K\"uhn, Theodore Molla, Deryk Osthus and Amelia Taylor}
\thanks{The research leading to these results was partially supported by the Simons Foundation, Grant no. 283194 (Louis DeBiasio), as well as by the European Research Council
under the European Union's Seventh Framework Programme (FP/2007--2013) / ERC Grant
Agreements no. 258345 (D.~K\"uhn) and 306349 (D.~Osthus).}

\newtheorem{firstthm}{Proposition}[section]
\newtheorem{theorem}[firstthm]{Theorem}
\newtheorem{prop}[firstthm]{Proposition}
\newtheorem{lemma}[firstthm]{Lemma}

\newdimen\margin   
\def\textno#1&#2\par{%
   \margin=\hsize
   \advance\margin by -4\parindent
          \setbox1=\hbox{\sl#1}%
   \ifdim\wd1 < \margin
      $$\box1\eqno#2$$%
   \else
      \bigbreak
      \hbox to \hsize{\indent$\vcenter{\advance\hsize by -3\parindent
      \it\noindent#1}\hfil#2$}%
      \bigbreak
   \fi}

\begin{document}

\def\COMMENT#1{}
\def\TASK#1{}

\def\eps{{\varepsilon}}
\newcommand{\ex}{\mathbb{E}}
\newcommand{\pr}{\mathbb{P}}
\newcommand{\cA}{\mathcal{A}}
\newcommand{\cB}{\mathcal{B}}
\newcommand{\cS}{\mathcal{S}}
\newcommand{\cF}{\mathcal{F}}
\newcommand{\cC}{\mathcal{C}}
\newcommand{\cP}{\mathcal{P}}
\newcommand{\cQ}{\mathcal{Q}}
\newcommand{\cR}{\mathcal{R}}
\newcommand{\cK}{\mathcal{K}}
\newcommand{\cD}{\mathcal{D}}
\newcommand{\cI}{\mathcal{I}}
\newcommand{\cV}{\mathcal{V}}
\newcommand{\1}{{\bf 1}_{n\not\equiv \delta}}
\newcommand{\eul}{{\rm e}}

\begin{abstract}  \noindent
Let $n$ be sufficiently large and suppose that $G$ is a digraph on $n$ vertices where every vertex has in- and outdegree at least $n/2$. We show that $G$ contains every orientation of a Hamilton cycle except, possibly, the antidirected one. The antidirected case was settled by DeBiasio and Molla, where the threshold is $n/2+1$. Our result is best possible and improves on an approximate result by H\"aggkvist and Thomason.
\end{abstract}

\maketitle

\section{Introduction}\label{sec:intro}

A classical result on Hamilton cycles is Dirac's theorem \cite{dirac} which states that if $G$ is a graph on $n \geq 3$ vertices with minimum degree $\delta(G) \geq n/2$, then $G$ contains a Hamilton cycle. Ghouila-Houri \cite{ghouila} proved an analogue of Dirac's theorem for digraphs  which guarantees that any digraph of minimum semidegree at least $n/2$ contains a consistently oriented Hamilton cycle (where the minimum semidegree $\delta^0(G)$ of a digraph $G$ is the minimum of all the in- and outdegrees of the vertices in $G$). In \cite{keevko}, Keevash, K\"uhn and Osthus proved a version of this theorem for oriented graphs. Here the minimum semidegree threshold turns out to be $\delta^0(G) \geq (3n-4)/8$. (In a digraph we allow two edges of opposite orientations between a pair or vertices, in an oriented graph at most one edge is allowed between any pair of vertices.)

Instead of asking for consistently oriented Hamilton cycles in an oriented graph or digraph, it is natural to consider different orientations of a Hamilton cycle. For example, Thomason \cite{thom} showed that every sufficiently large strongly connected tournament contains every orientation of a Hamilton cycle. H\"aggkvist and Thomason \cite{hagthom2} proved an approximate version of Ghouila-Houri's theorem for arbitrary orientations of Hamilton cycles. They showed that a minimum semidegree of $n/2+n^{5/6}$ ensures the existence of an arbitrary orientation of a Hamilton cycle in a digraph. This improved a result of Grant \cite{grant} for antidirected Hamilton cycles. The exact threshold in the antidirected case was obtained by DeBiasio and Molla \cite{debmol}, here the threshold is $\delta^0(G) \geq n/2+1$, i.e., larger than in Ghouila-Houri's theorem. In Figure~\ref{fig:cai}, we give two digraphs $G$ on $2m$ vertices which satisfy $\delta^0(G) = m$ and have no antidirected Hamilton cycle, showing that this bound is best possible. (The first of these examples is already due to Cai \cite{cai}.)
\begin{figure}[h]
\centering
\includegraphics[scale=0.2]{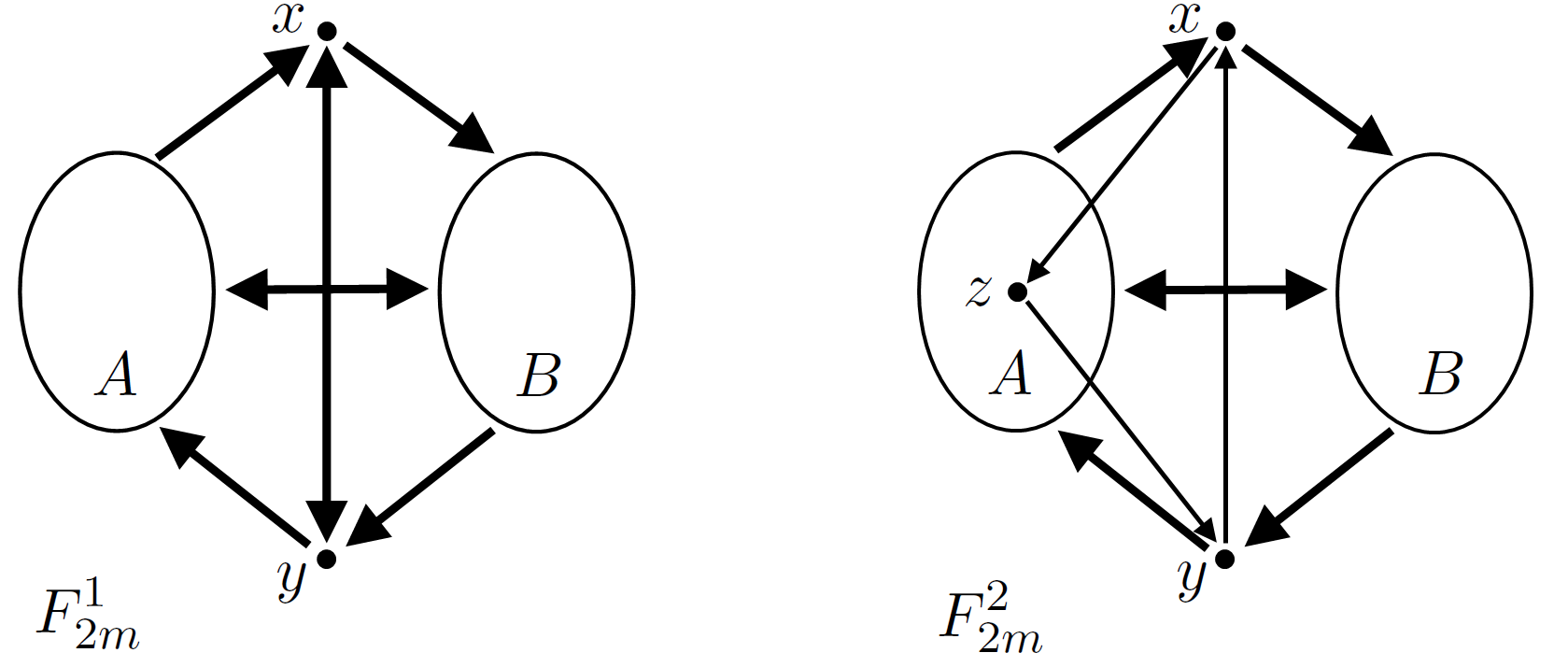}
\caption{In digraphs $F_{2m}^1$ and $F_{2m}^2$, $A$ and $B$ are independent sets of size $m-1$ and bold arrows indicate that all possible edges are present in the directions shown.}\label{fig:cai}
\end{figure}
\begin{theorem}[DeBiasio \& Molla, \cite{debmol}]\label{thm:antidirected}
There exists an integer $m_0$ such that the following hold for all $m\geq m_0$. Let $G$ be a digraph on $2m$ vertices. If $\delta^0(G) \geq m$, then $G$ contains an antidirected Hamilton cycle, unless $G$ is isomorphic to $F_{2m}^1$ or $F_{2m}^2$. In particular, if $\delta^0(G) \geq m+1$, then $G$ contains an antidirected Hamilton cycle.
\end{theorem}

In this paper, we settle the problem by completely determining the exact threshold for arbitrary orientations. We show that a minimum semidegree of $n/2$ suffices if the Hamilton cycle is not antidirected. This bound is best possible by the extremal examples for Ghouila-Houri's theorem, i.e., if $n$ is even, the digraph consisting of two disjoint complete digraphs on $n/2$ vertices and, if $n$ is odd, the complete bipartite digraph with vertex classes of size $(n-1)/2$ and $(n+1)/2$.
\begin{theorem}\label{thm:main}
There exists an integer $n_0$ such that the following holds.  Let $G$ be a digraph on $n\geq n_0$ vertices with $\delta^0(G) \geq n/2$. If $C$ is any orientation of a cycle on $n$ vertices which is not antidirected, then $G$ contains a copy of $C$.
\end{theorem}

Kelly \cite{kelly} proved an approximate version of Theorem~\ref{thm:main} for oriented graphs. He showed that the semidegree threshold for an arbitrary orientation of a Hamilton cycle in an oriented graph is $3n/8 +o(n)$. It would be interesting to obtain an exact version of this result. Further related problems on digraph Hamilton cycles are discussed in \cite{kosurvey}.

\section{Proof sketch}\label{sec:sketch}

The proof of Theorem~\ref{thm:main} utilizes the notion of robust expansion which has been very useful in several settings recently. Roughly speaking, a digraph $G$ is a robust outexpander if every vertex set $S$ of reasonable size has an outneighbourhood which is at least a little larger than $S$ itself, even if we delete a small proportion of the edges of $G$. A formal definition of robust outexpansion is given in Section~\ref{sec:tools}. In Lemma~\ref{lem:structure}, we observe that any graph satisfying the conditions of Theorem~\ref{thm:main} must be a robust outexpander or have a large set which does not expand, in which case we say that $G$ is $\eps$-extremal.  Theorem~\ref{thm:main} was verified for the case when $G$ is a robust outexpander by Taylor \cite{msci} based on the approach of Kelly \cite{kelly}. This allows us to restrict our attention to the $\eps$-extremal case. We introduce three refinements of the notion of $\eps$-extremality: $ST$-extremal, $AB$-extremal and $ABST$-extremal. These are illustrated in Figure~\ref{fig:abstextremal}, the arrows indicate that $G$ is almost complete in the directions shown. In each of these cases, we have that $|A|\sim|B|$ and $|S|\sim|T|$. If $G$ is $ST$-extremal, then the sets $A$ and $B$ are almost empty and so $G$ is close to the digraph consisting of two disjoint complete digraphs on $n/2$ vertices. If $G$ is $AB$-extremal, then the sets $S$ and $T$ are almost empty and so in this case $G$ is close to the complete bipartite digraph with vertex classes of size $n/2$ (thus both digraphs in Figure~\ref{fig:cai} are $AB$-extremal). Within each of these cases, we further subdivide the proof depending on how many changes of direction the desired Hamilton cycle has. Note that in the directed setting the set of extremal structures is much less restricted than in the undirected setting (in the undirected case, it is well known that all the near extremal graphs are close to the complete bipartite graph $K_{n/2,n/2}$ or two disjoint cliques on $n/2$ vertices).

\begin{figure}[h]
\centering
\includegraphics[scale=0.25]{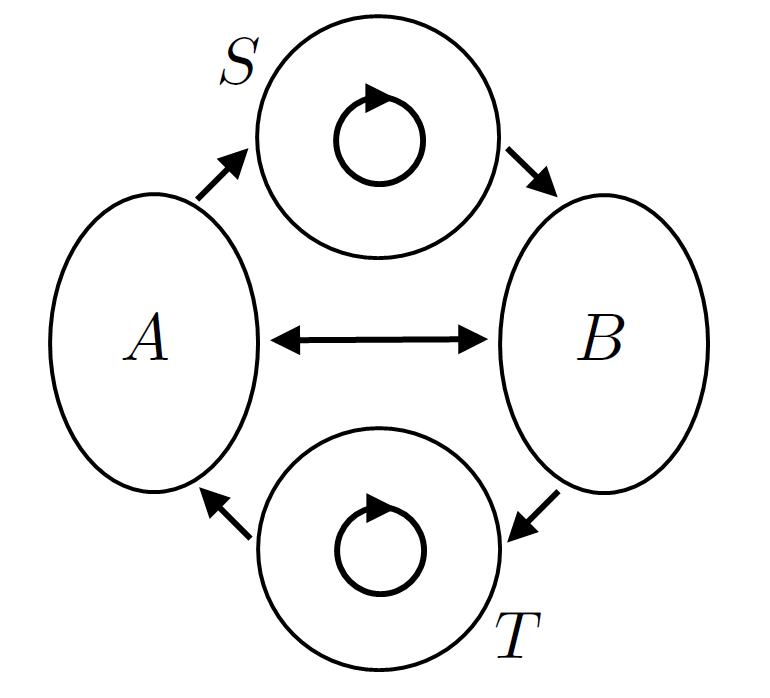}
\caption{An $ABST$-extremal graph. When $G$ is $AB$-extremal, the sets $S$ and $T$ are almost empty and when $G$ is $ST$-extremal the sets $A$ and $B$ are almost empty.}\label{fig:abstextremal}
\end{figure}

The main difficulty in each of the cases is covering the exceptional vertices, i.e., those vertices with low in- or outdegree in the vertex classes where we would expect most of their neighbours to lie. When $G$ is $AB$-extremal, we also consider the vertices in $S\cup T$ to be exceptional and, when $G$ is $ST$-extremal, we consider the vertices in $A\cup B$ to be exceptional. In each case we find a short path $P$ in $G$ which covers all of these exceptional vertices. When the cycle $C$ is close to being consistently oriented, we cover these exceptional vertices by short consistently oriented paths and when $C$ has many changes of direction, we will map sink or source vertices in $C$ to these exceptional vertices (here a sink vertex is a vertex of indegree two and a source vertex is a vertex of outdegree two).

An additional difficulty is that in the $AB$- and $ABST$-extremal cases we must ensure that the path $P$ leaves a balanced number of vertices in $A$ and $B$ uncovered. Once we have found $P$ in $G$, the remaining vertices of $G$ (i.e., those not covered by $P$) induce a balanced almost complete bipartite digraph and one can easily embed the remainder of $C$ using a bipartite version of Dirac's theorem. When $G$ is $ST$-extremal, our aim will be to split the cycle $C$ into two paths $P_S$ and $P_T$ and embed $P_S$ into the digraph $G[S]$ and $P_T$ into $G[T]$. So a further complication in this case is that we need to link together $P_S$ and $P_T$ as well as covering all vertices in $A\cup B$. 

This paper is organised as follows. Sections~\ref{sec:notation} and \ref{sec:tools} introduce the notation and tools which will be used throughout this paper.  In Section~\ref{sec:structure} we describe the structure of an $\eps$-extremal digraph and formally define what it means to be $ST$-, $AB$- or $ABST$-extremal. The remaining sections prove Theorem~\ref{thm:main} in each of these three cases: we consider the $ST$-extremal case in Section~\ref{sec:ST}, the $AB$-extremal case in Section~\ref{sec:AB} and the $ABST$-extremal case in Section~\ref{sec:ABST}.

\section{Notation}\label{sec:notation}

Let $G$ be a digraph on $n$ vertices. We will write $xy \in E(G)$ to indicate that $G$ contains an edge oriented from $x$ to $y$. If $G$ is a digraph and $x\in V(G)$, we will write $N^+_G(x)$ for the \emph{outneighbourhood} of $x$ and $N^-_G(x)$ for the \emph{inneighbourhood} of $x$. We define $d^+_G(x):=|N^+_G(x)|$ and $d^-_G(x):=|N^-_G(x)|$. We will write, for example, $d^\pm_G(x)\geq a$ to mean $d^+_G(x), d^-_G(x) \geq a$. We sometimes omit the subscript $G$ if this is unambiguous. We let $\delta^0(G):=\min\{d^+(x), d^-(x):x\in V(G)\}$. If $A\subseteq V(G)$, we let $d^+_A(x):=|N^+_G(x)\cap A|$ and define $d^-_A(x)$ and $d^\pm_A(x)$ similarly. We say that $x\in V(G)$ is a \emph{sink vertex} if $d^+(x)=0$ and a \emph{source vertex} if $d^-(x)=0$.

Let $A, B \subseteq V(G)$ and $xy\in E(G)$. If $x\in A$ and $y\in B$ we say that $xy$ is an \emph{$AB$-edge}. We write $E(A,B)$ for the set of all $AB$-edges and we write $E(A)$ for $E(A,A)$. We let $e(A, B):=|E(A,B)|$ and $e(A):=|E(A)|$. We write $G[A,B]$ for the digraph with vertex set $A\cup B$ and edge set $E(A,B)\cup E(B,A)$ and we write $G[A]$ for the digraph with vertex set $A$ and edge set $E(A)$. We say that a path $P=x_1x_2\dots x_q$ is an \emph{$AB$-path} if $x_1\in A$ and $x_q\in B$. If $x_1,x_q\in A$, we say that $P$ is an \emph{$A$-path}. If $A \subseteq V(P)$, we say that $P$ \emph{covers} $A$. If $\cP$ is a collection of paths, we write $V(\cP)$ for $\bigcup_{P\in\cP}V(P)$.

Let $P=x_1x_2\dots x_q$ be a path. The \emph{length} of $P$ is the number of its edges. Given sets $X_1, \dots, X_q\subseteq V(G)$, we say that $P$ has \emph{form} $X_1X_2\dots X_q$  if $x_i\in X_i$ for $i=1,2, \dots, q$. We will use the following abbreviation
$$(X)^k:=\underbrace{XX\dots X}_{k\text{ times}}.$$
We will say that $P$ is a \emph{forward} path of the form $X_1X_2\dots X_q$ if $P$ has form $X_1X_2\dots X_q$ and $x_ix_{i+1}\in E(P)$ for all $i=1,2,\dots, q-1$. Similarly, $P$ is a \emph{backward} path of the form $X_1X_2\dots X_q$ if $P$ has form $X_1X_2\dots X_q$ and $x_{i+1}x_{i}\in E(P)$ for all $i=1,2,\dots, q-1$.

A digraph $G$ is \emph{oriented} if it is an orientation of a simple graph (i.e., if there are no $x,y\in V(G)$ such that $xy, yx\in E(G)$). Suppose that $C=(u_1u_2\dots u_n)$ is an oriented cycle. We let $\sigma(C)$ denote the number of sink vertices in $C$. We will write $(u_iu_{i+1} \dots u_j)$ or $(u_iCu_j)$ to denote the subpath of $C$ from $u_i$ to $u_j$. In particular, $(u_iu_{i+1})$ may represent the edge $u_iu_{i+1}$ or $u_{i+1}u_i$. Given edges $e=(u_i,u_{i+1})$ and $f=(u_j,u_{j+1})$, we write $(eCf)$ for the path $(u_iCu_{j+1})$. We say that an edge $(u_iu_{i+1})$ is a \emph{forward edge} if $(u_iu_{i+1})=u_iu_{i+1}$ and a \emph{backward edge} if $(u_iu_{i+1})=u_{i+1}u_i$. We say that a cycle is \emph{consistently oriented} if all of its edges are oriented in the same direction (forward or backward). We define a consistently oriented subpath $P$ of $C$ in the same way. We say that $P$ is \emph{forward} if it consists of only forward edges and \emph{backward} if it consists of only backward edges. A collection of subpaths of $C$ is \emph{consistent} if they are all forward paths or if they are all backward paths. We say that a path or cycle is \emph{antidirected} if it contains no consistently oriented subpath of length two.

Given $C$ as above, we define $d_C(u_i,u_j)$ to be the length of the path $(u_iCu_j)$ (so, for example, $d_C(u_1, u_n)=n-1$ and $d_C(u_n, u_1)=1$).
For a subpath $P=(u_iu_{i+1}\dots u_k)$ of $C$, we call $u_i$ the \emph{initial} vertex of $P$ and $u_k$ the \emph{final vertex}. We write $(u_jP):=(u_ju_{j+1}\dots u_k)$ and $(Pu_j):=(u_iu_{i+1}\dots u_j)$. If $P_1$ and $P_2$ are subpaths of $C$, we define $d_C(P_1,P_2):=d_C(v_1,v_2)$, where $v_i$ is the initial vertex $P_i$. In particular, we will use this definition when one or both of $P_1,P_2$ are edges. Suppose $P_1, P_2, \dots, P_k$ are internally disjoint subpaths of $C$ such that the final vertex of $P_i$ is the initial vertex of $P_{i+1}$ for $i=1, \dots, k-1$. Let $x$ denote the initial vertex of $P_1$ and $y$ denote the final vertex of $P_k$. If $x\neq y$, we write $(P_1P_2\dots P_k)$ for the subpath of $C$ from $x$ to $y$. If $x=y$, we sometimes write $C=(P_1P_2\dots P_k)$.

We will also make use of the following notation: $a \ll b$. This means that we can find an increasing function $f$ for which all of the conditions in the proof are satisfied whenever $a \leq f(b)$. It is equivalent to setting $a := \min \{f_1(b), f_2(b), \dots, f_k(b)\}$, where each $f_i(b)$ corresponds to the maximum value of $a$ allowed in order that the corresponding argument in the proof holds. However, in order to simplify the presentation, we will not determine these functions explicitly.

\section{Tools}\label{sec:tools}

\subsection{Hamilton cycles in dense graphs and digraphs}

We will use the following standard results concerning Hamilton paths and cycles. Theorem~\ref{thm:moon} is a bipartite version of Dirac's theorem. Proposition~\ref{prop:completepath} is a simple consequence of Dirac's theorem and this bipartite version.

\begin{theorem}[Moon \& Moser, \cite{moon}]\label{thm:moon}
Let $G=(A,B)$ be a bipartite graph with $|A|=|B|=n$. If $\delta(G) \geq n/2+1$, then $G$ contains a Hamilton cycle.
\end{theorem}

\begin{prop}\label{prop:completepath}
\begin{enumerate}[\rm(i)]
\item Let $G$ be a digraph on $n$ vertices with $\delta^0(G)\geq 7n/8$. Let $x,y\in V(G)$ be distinct. Then $G$ contains a Hamilton path of any orientation between $x$ and $y$.
\item Let $m \geq 10$ and $G=(A,B)$ be a bipartite digraph with $|A|=m+1$ and $|B|=m$. Suppose that $\delta^0(G)\geq (7m+2)/8$. Let $x,y \in A$. Then $G$ contains a Hamilton path of any orientation between $x$ and $y$.
\end{enumerate}
\end{prop}

\begin{proof}
To prove (i), we define an undirected graph $G'$ on the vertex set $V(G)$ where $uv \in E(G')$ if and only if $uv, vu \in E(G)$. Let $G''$ be the graph obtained from $G'$ by contracting the vertices $x$ and $y$ to a single vertex $x'$ with $N_{G''}(x'):=N_{G'}(x) \cap N_{G'}(y)$. Note that%
\COMMENT{$d_{G'}(v)\geq 7n/8-(n/8-1)=3n/4+1$ for all $v\in V(G)$. So for $v\neq x,y$, we have $d_{G''}(x')\geq 3n/4\geq (n-1)/2$. Also, $d_{G''}(x')\geq$ (no. neighbours of $x$ in $G'$)$-$(no. non-neighbours of $y$ in $G'$)$-1$, where $-1$ is in case $y$ is a neighbour of $x$. So $d_{G''}(x')\geq (3n/4+1)-(n/4-2)-1=n/2+2\geq (n-1)/2$.} 
$$\delta(G'') \geq (n-1)/2= |G''|/2.$$
Hence $G''$ has a Hamilton cycle by Dirac's theorem. This corresponds to a Hamilton path of any orientation between $x$ and $y$ in $G$.

For (ii), we proceed in the same way, using Theorem~\ref{thm:moon} instead of Dirac's theorem.%
\COMMENT{We define an undirected bipartite graph $G'$ with vertex classes $A$ and $B$ where for all $u \in A, v\in B$, $uv \in E(G')$ if and only if $uv, vu \in E(G)$. Obtain the graph $G''$ by contracting the vertices $x$ and $y$ to a single vertex $x'$ with $N_{G''}(x'):=N_{G'}(x)\cap N_{G'}(y)$. Note that the resulting graph has vertex classes $A',B'$ of equal size, $m$. For any $v\in A$, we have $d_{G'}(v) \geq 2(7m+2)/8-m =(3m+2)/4$. So for $v\in A$ with $v \neq x,y$, we have $d_{G''}(v) \geq (3m+2)/4$. $d_{G''}(x') \geq 2(3m+2)/4-m = m/2+1$. For any $v\in B$ we have $d_{G'}(v) \geq 2(7m+2)/8-(m+1) =(3m-2)/4$ and so $d_{G''}(v) \geq (3m-2)/4-1 =(3m-6)/4$ (where $-1$ accounts for contraction of $x,y$). Note that $(3m-6)/4\geq m/2+1$ for all $m \geq 10$. So $G''$ satisfies $\delta(G'') \geq m/2+1.$
Then, by Theorem~\ref{thm:moon}, $G''$ has a Hamilton cycle. This Hamilton cycle uses only edges that are present in both directions in $G$. This allows us to find a Hamilton path of any orientation between $x$ and $y$ in $G$.}
\end{proof}

\subsection{Robust expanders}\label{subsec:robexp}

Let $0< \nu \leq \tau <1$, let $G$ be a digraph on $n$ vertices and let $S \subseteq V(G)$. The \emph{$\nu$-robust outneighbourhood} $RN_{\nu,G}^+(S)$ of $S$ is the set of all those vertices $x \in V(G)$ which have at least $\nu n$ inneighbours in $S$. $G$ is called a \emph{robust $(\nu,\tau)$-outexpander} if $|RN_{\nu,G}^+(S)| \geq |S|+\nu n$ for all $S \subseteq V(G)$ with $\tau n < |S| < (1- \tau)n$.

Recall from Section~\ref{sec:intro} that Kelly \cite{kelly} showed that any sufficiently large oriented graph with minimum semidegree at least $(3/8+\alpha)n$ contains any orientation of a Hamilton cycle. It is not hard to show that any such oriented graph is a robust outexpander (see \cite{kellyexact78}). In fact, in \cite{kelly}, Kelly observed that his arguments carry over to robustly expanding digraphs of linear degree. Taylor \cite{msci} has verified that this is indeed the case, proving the following result.
\begin{theorem}[\cite{msci}]\label{thm:robust}
Suppose $1/n \ll \nu \leq \tau \ll \eta <1$. Let $G$ be a digraph on $n$ vertices with $\delta^0(G) \geq \eta n$ and suppose $G$ is a robust $(\nu, \tau)$-outexpander. If $C$ is any orientation of a cycle on $n$ vertices, then $G$ contains a copy of $C$.
\end{theorem}

\subsection{Structure}\label{sec:structure}

Let $\eps >0$ and $G$ be a digraph on $n$ vertices. We say that $G$ is \emph{$\eps$-extremal} if there is a partition $A,B,S,T$ of its vertices into sets of sizes $a,b,s,t$ such that $|a-b|,|s-t|\leq 1$ and $e(A\cup S, A\cup T) < \eps n^2$.

The following lemma describes the structure of a graph which satisfies the conditions of Theorem~\ref{thm:main}.

\begin{lemma}\label{lem:structure}
Suppose
$0<1/n \ll \nu \ll \tau, \eps < 1$
and let $G$ be a digraph on $n$ vertices with
$ \delta^0(G) \geq n/2.$
Then $G$ satisfies one of the following:
\begin{enumerate}[\rm(i)]
	\item $G$ is $\eps$-extremal;
	\item $G$ is a robust $(\nu, \tau)$-outexpander.
\end{enumerate}
\end{lemma}

\begin{proof}
Suppose that $G$ is not a robust $(\nu, \tau)$-outexpander. Then there is a set $X \subseteq V(G)$ with $\tau n \leq |X| \leq (1-\tau)n$ and $|RN_{\nu,G}^+(X)|<|X|+\nu n$. Define $RN^+ := RN_{\nu,G}^+(X)$. We consider the following cases:

\medskip

\noindent \textbf{Case 1: } \emph{$\tau n \leq |X| \leq (1/2- \sqrt \nu)n$.}

We have $$ |X|n/2 \leq e(X, RN^+) + e(X, \overline {RN^+}) \leq |X||RN^+| + \nu n^2 \leq |X|(|RN^+| + \nu n/\tau),$$
so $|RN^+| \geq (1/2-\nu/\tau)n \geq |X| +\nu n,$
which gives a contradiction.

\medskip

\noindent \textbf{Case 2: } \emph{$(1/2+\nu)n \leq |X| \leq (1-\tau)n$.}

For any $v \in V(G)$ we note that $d^-_X(v) \geq \nu n.$
Hence $|RN^+|=|G| \geq |X| +\nu n$, a contradiction.

\medskip

\noindent \textbf{Case 3: } \emph{$(1/2- \sqrt \nu)n < |X| < (1/2+ \nu)n$.}

Suppose that $|RN^+|<(1/2-3\nu) n$. Since $\delta^0(G) \geq n/2$, each vertex in $X$ has more than $3\nu n$ outneighbours in $\overline{RN^+}$. Thus, there is a vertex $v \not\in RN^+$ with more than $3\nu n |X|/n > \nu n$ inneighbours in $X$, which is a contradiction. Therefore,
\begin{equation}\label{eqn:rn+}
	(1/2-3\nu) n \leq |RN^+|< |X|+\nu n < (1/2+ 2\nu)n.
\end{equation}
Write $A_0:=X\setminus RN^+$, $B_0:=RN^+\setminus X$, $S_0:=X\cap RN^+$ and $T_0 := \overline X\cap \overline{RN^+}$. Let $a_0,b_0,s_0,t_0$, respectively, denote their sizes. Note that $|X|=a_0+s_0$, $|RN^+|=b_0+s_0$ and $a_0+b_0+s_0+t_0=n$. It follows from (\ref{eqn:rn+}) and the conditions of Case 3 that
$$(1/2-\sqrt\nu) n \leq a_0+s_0, b_0+t_0, b_0+s_0, a_0+t_0 \leq (1/2+\sqrt \nu) n$$
 and so $|a_0-b_0|, |s_0-t_0| \leq 2\sqrt\nu n$. Note that
$$e(A_0\cup S_0, A_0\cup T_0) = e(X,\overline{RN^+})<\nu n^2.$$ By moving at most $\sqrt\nu n$ vertices between the sets $A_0$ and $B_0$ and $\sqrt \nu n$ between the sets $S_0$ and $T_0$, we obtain new sets $A,B,S,T$ of sizes $a,b,s,t$ satisfying $|a-b|, |s-t| \leq 1$ and $e(A \cup S, A\cup T) \leq \eps n^2$. So $G$ is $\eps$-extremal.
\end{proof}

\subsection{Refining the notion of $\eps$-extremality}\label{subsec:refine}

Let $n\in \mathbb{N}$ and $\eps, \eps_1, \eps_2, \eps_3, \eps_4, \eta_1, \eta_2, \tau$ be positive constants satisfying
$$1/n \ll \eps \ll \eps_1 \ll \eps_2\ll \eta_1 \ll \tau \ll \eps_3 \ll \eps_4 \ll \eta_2 \ll 1.$$
We now introduce three refinements of $\eps$-extremality. (The constants $\eps_2$ and $\eps_4$ do not appear in these definitions but will be used at a later stage in the proof so we include them here for clarity.) Let $G$ be a digraph on $n$ vertices.

Firstly, we say that $G$ is \emph{$ST$-extremal} if there is a partition $A,B,S,T$ of $V(G)$ into sets of sizes $a,b,s,t$  such that:
\begin{enumerate}[(P1)]
\item $a\leq b$, $s\leq t$;\label{P*}
\item $\lfloor n/2 \rfloor -\eps_3 n \leq s,t \leq \lceil n/2 \rceil + \eps_3 n$;\label{P1}
\item $\delta^0(G[S]), \delta^0(G[T])  \geq \eta_2 n$;\label{P2}
\item $d^\pm_S(x) \geq n/2-\eps_3 n$ for all but at most $\eps_3 n$ vertices $x\in S$;\label{P3}
\item $d^\pm_T(x) \geq n/2-\eps_3 n$ for all but at most $\eps_3 n$ vertices $x\in T$;\label{P4}
\item $a+b \leq \eps_3 n$;\label{P5}
\item $d^-_T(x), d^+_S(x) > n/2-3\eta_2 n$ and $d^-_S(x), d^+_T(x)\leq 3\eta_2 n$ for all $x \in A$;\label{P6}
\item $d^-_S(x), d^+_T(x)> n/2-3\eta_2 n$ and $d^-_T(x), d^+_S(x) \leq 3\eta_2 n$ for all $x \in B$.\label{P7}
\end{enumerate}

Secondly, we say that $G$ is \emph{$AB$-extremal} if there is a partition $A,B,S,T$ of $V(G)$ into sets of sizes $a,b,s,t$  such that:
\begin{enumerate}[(Q1)]
\item $a\leq b$, $s\leq t$;\label{Q*}
\item $\lfloor n/2 \rfloor -\eps_3 n \leq a, b \leq \lceil n/2 \rceil + \eps_3 n$;\label{Q1}
\item $\delta^0(G[A,B])\geq n/50$;\label{Q2}
\item $d^\pm_B(x)\geq n/2-\eps_3 n$ for all but at most $\eps_3 n$ vertices $x\in A$;\label{Q3}
\item $d^\pm_A(x)\geq n/2-\eps_3 n$ for all but at most $\eps_3 n$ vertices $x\in B$;\label{Q4}
\item $s+t \leq \eps_3 n$;\label{Q5}
\item $d^-_A(x), d^+_B(x) \geq n/50$ for all $x \in S$;\label{Q6}
\item $d^-_B(x), d^+_A(x) \geq n/50$ for all $x \in T$;\label{Q7}
\item if $a<b$, $d^\pm_B(x) < n/20$ for all $x\in B$; $d^-_B(x)< n/20$ for all $x \in S$ and $d^+_B(x)< n/20$ for all $x \in T$.\label{Q8}%
\COMMENT{Only used in Prop.~\ref{prop:balance}. Cannot lose any conditions here because I need to be able to swap $S$ and $T$ to get $s\leq t$, if necessary.}
\end{enumerate}

Thirdly, we say that $G$ is \emph{$ABST$-extremal} if there is a partition $A,B,S,T$ of $V(G)$ into sets of sizes $a,b,s,t$  such that:
\begin{enumerate}[(R1)]
	\item $a\leq b$, $s\leq t$;\label{R*}
	\item $a,b,s,t \geq \tau n$;\label{R1}
	\item $|a-b|, |s-t| \leq \eps_1n$;\label{R2}
	\item $\delta^0(G[A,B])\geq \eta_1 n$;\label{R3}
	\item $d^+_{B\cup S}(x), d^-_{A\cup S}(x) \geq \eta_1 n$ for all $x \in S$;\label{R4}
	\item $d^+_{A\cup T}(x), d^-_{B\cup T}(x) \geq \eta_1 n$ for all $x \in T$;\label{R5}
	\item $d^\pm_B(x) \geq b-\eps^{1/3} n$ for all but at most $\eps_1n$ vertices $x\in A$;\label{R6}
	\item $d^\pm_A(x) \geq a-\eps^{1/3} n$ for all but at most $\eps_1n$ vertices $x\in B$;\label{R7}
	\item $d^+_{B\cup S}(x)\geq b+s-\eps^{1/3} n$ and $d^-_{A\cup S}(x) \geq a+s-\eps^{1/3} n$ for all but at most $\eps_1n$ vertices $x\in S$;\label{R8}
	\item $d^+_{A\cup T}(x)\geq a+t-\eps^{1/3} n$ and $d^-_{B\cup T}(x) \geq b+t-\eps^{1/3} n$ for all but at most $\eps_1n$ vertices $x\in T$.\label{R9}
\end{enumerate}

\begin{prop}\label{prop:structure2}
Suppose
$$1/n \ll \eps \ll \eps_1 \ll \eta_1 \ll \tau \ll \eps_3 \ll \eta_2 \ll 1$$
and $G$ is an $\eps$-extremal digraph on $n$ vertices with $\delta^0(G) \geq n/2$. Then there is a partition of $V(G)$ into sets $A,B,S,T$ of sizes $a,b,s,t$ satisfying one of the following:
\begin{itemize}
\item (P\ref{P1})--(P\ref{P7});
\item (Q\ref{Q1})--(Q\ref{Q8}) with $a\leq b$; 
\item (R\ref{R1})--(R\ref{R9}).
\end{itemize}
\end{prop}

\begin{proof}
Consider a partition $A_0,B_0,S_0,T_0$ of $V(G)$ into sets of sizes $a_0,b_0,s_0,t_0$ such that $|a_0-b_0|,|s_0-t_0|\leq 1$ and $e(A_0\cup S_0, A_0\cup T_0) < \eps n^2$. 
Define 
\begin{align*}
&X_1:=\{x\in A_0 \cup S_0: d^+_{B_0 \cup S_0}(x) < n/2-\sqrt\eps n\},\\
&X_2:=\{x\in A_0 \cup T_0: d^-_{B_0 \cup T_0}(x) < n/2-\sqrt\eps n\},\\
&X_3:=\{x\in B_0 \cup T_0: d^+_{A_0 \cup T_0}(x) < n/2-\sqrt\eps n\},\\
&X_4:=\{x\in B_0 \cup S_0: d^-_{A_0 \cup S_0}(x) < n/2-\sqrt\eps n\}
\end{align*}
and let $X:=\bigcup_{i=1}^4 X_i$. We now compute an upper bound for $|X|$. 
Each vertex $x\in X_1$ has $d^+_{A_0\cup T_0}(x)> \sqrt\eps n$, so $|X_1|\leq \eps n^2/ \sqrt\eps n=\sqrt\eps n$.  Also, each vertex $x\in X_2$ has $d^-_{A_0\cup S_0}(x) > \sqrt\eps n$, so $|X_2|\leq \sqrt\eps n$. 
Observe that
\begin{align*}
|A_0\cup T_0|n/2-\eps n^2 &\leq e(B_0\cup T_0,A_0\cup T_0)\\
&\leq (n/2-\sqrt\eps n)|X_3|+|A_0\cup T_0|(|B_0\cup T_0|-|X_3|)
\end{align*}
which gives
$$|X_3|(|A_0\cup T_0|-n/2+\sqrt\eps n) \leq |A_0\cup T_0|(|B_0\cup T_0|-n/2)+\eps n^2\leq 2\eps n^2.$$
So $|X_3| \leq 2\eps n^2/(\sqrt\eps n/2)=4\sqrt\eps n$. Similarly, we find that $|X_4| \leq 4\sqrt\eps n$.%
\COMMENT{$|A_0\cup S_0|n/2-\eps n^2 \leq e(A_0\cup S_0, B_0 \cup S_0)
\leq (n/2-\sqrt\eps n)|X_4|+|A_0\cup S_0|(|B_0\cup S_0|-|X_4|)$}
 Therefore, $|X|\leq 10\sqrt\eps n$.

\medskip

\noindent \textbf{Case 1: } \emph{$a_0,b_0<2\tau n$.}

Let $Z:=X\cup A_0 \cup B_0$. Choose disjoint $Z_1, Z_2\subseteq Z$ so that $d^\pm_{S_0}(x) \geq 2\eta_2n$ for all $x\in Z_1$ and $d^\pm_{T_0}(x) \geq 2\eta_2n$ for all $x\in Z_2$ and $|Z_1\cup Z_2|$ is maximal. Let $S:=(S_0 \setminus X) \cup Z_1$ and $T:=(T_0\setminus X) \cup Z_2$. The vertices in $Z\setminus (Z_1\cup Z_2)$ can be partitioned into two sets $A$ and $B$ so that $d^+_S(x), d^-_T(x) \geq n/2-3\eta_2n$ for all $x\in A$ and $d^-_S(x), d^+_T(x) \geq n/2-3\eta_2n$ for all $x\in B$. The partition $A,B, S, T$ satisfies (P\ref{P1})--(P\ref{P7}).

\medskip

\noindent \textbf{Case 2: } \emph{$s_0,t_0<2\tau n$.}

Partition $X$ into four sets $Z_1, Z_2, Z_3, Z_4$ so that $d^\pm_{B_0}(x)\geq n/5$ for all $x\in Z_1$; $d^\pm_{A_0}(x)\geq n/5$ for all $x\in Z_2$; $d^+_{B_0}(x), d^-_{A_0}(x)\geq n/5$ for all $x\in Z_3$ and $d^-_{B_0}(x), d^+_{A_0}(x) \geq n/5$ for all $x\in Z_4$. Then set $A_1:=(A_0\setminus X)\cup Z_1$, $B_1:=(B_0\setminus X)\cup Z_2$.

Assume, without loss of generality, that $|A_1|\leq |B_1|$. To ensure that the vertices in $B$ satisfy (Q\ref{Q8}), choose disjoint sets $B', B''\subseteq B_1$ so that $|B'\cup B''|$ is maximal subject to: $|B'\cup B''|\leq |B_1|-|A_1|$, $d^+_{B_1}(x)\geq n/20$ for all $x\in B'$ and $d^-_{B_1}(x)\geq n/20$ for all $x\in B''$. Set $B:= B_1 \setminus (B' \cup B'')$, $S_1 :=(S_0\setminus X)\cup Z_3\cup B'$ and $T_1:=(T_0\setminus X)\cup Z_4\cup B''$. To ensure that the vertices in $S\cup T$ satisfy (Q\ref{Q8}), choose sets  $S'\subseteq S_1, T'\subseteq T_1$ which are maximal subject to:   $|S'|+|T'|\leq |B|-|A_1|$, $d^\pm_{B}(x) \geq n/20$ for all $x\in S'$ and $d^\pm_{B}(x) \geq n/20$ for all $x\in T'$. We define $A:=A_1 \cup S'\cup T'$, $S:=S_1\setminus S'$ and $T:=T_1\setminus T'$. Then $a\leq b$ and (Q\ref{Q1})--(Q\ref{Q8}) hold.%
\COMMENT{Note that $|X|\leq 10\sqrt\eps n$ so we move at most $\eps_3^2n$ vertices to get to $A,B,S,T$ from $A_0,B_0, S_0, T_0$. (Q\ref{Q1}) is clear. For (Q\ref{Q2})--(Q\ref{Q4}), note that we defined $Z_1, Z_2,S', T'$ carefully so that $\delta^0(G[A,B]) \geq n/50$. For (Q\ref{Q5}), note that $s_0+t_0<4\tau n$ so $s+t \leq 4\tau n+\eps_3^2 n\leq \eps_3 n$. (Q\ref{Q6}) and (Q\ref{Q7}) follow from the definitions of $A_1, B_1, S_1,T_1$.}

\medskip

\noindent \textbf{Case 3: } \emph{$a_0,b_0,s_0, t_0\geq 2\tau n-1$.}

The case conditions imply $a_0,b_0,s_0, t_0<n/2-\tau n$. Then, since $\delta^0(G) \geq n/2$, each vertex must have at least $2\eta_1 n$ inneighbours in at least two of the sets $A_0, B_0, S_0, T_0$. The same holds when we consider outneighbours instead.  So we can partition the vertices in $X$ into sets $Z_1, Z_2, Z_3, Z_4$ so that: $d^\pm_{B_0}(x) \geq 2\eta_1 n$ for all $x\in Z_1$; $d^\pm_{A_0}(x) \geq 2\eta_1 n$ for all $x\in Z_2$; $d^+_{B_0\cup S_0}(x), d^-_{A_0\cup S_0}(x) \geq 2\eta_1 n$ for all $x\in Z_3$ and $d^+_{A_0\cup T_0}(x), d^-_{B_0\cup T_0}(x) \geq 2\eta_1 n$ for all $x\in Z_4$.%
\COMMENT{To see that each $x$ can be put in at least one of $Z_1, Z_2, Z_3, Z_4$: Let $x\in V(G)$ and define $X^+:=\{X\in \{A_0,B_0,S_0,T_0\}: d^+_X(x) \geq 2 \eta_1 n\}$. Define $X^-$ similarly. Then $|X^+|,|X^-|\geq 2$. If $\exists X\in X^+\cap X^-$ then if $X=B_0$ can put $x$ in $Z_1$, $X=A_0$ can put $x$ in $Z_2$, $X=S_0$ put $x$ in $Z_3$ and $X=T_0$ put $x$ in $Z_4$. So assume $X^+\cap X^-=\emptyset$, so $X^+=\{A_0, B_0, S_0, T_0\} \setminus X^-$. Check for each. If $X^+=\{A_0, B_0\}$ or $\{A_0, S_0\}$ or $\{A_0, T_0\}$, can put $x$ in $Z_4$. If $X^+=\{B_0,S_0\}$ or $\{B_0,T_0\}$ or $\{S_0 , T_0\}$ put $x$ in $Z_3$.}
Let $A:=(A_0\setminus X)\cup Z_1$, $B:=(B_0\setminus X)\cup Z_2$, $S:=(S_0\setminus X)\cup Z_3$ and $T:=(T_0\setminus X)\cup Z_4$. This partition satisfies (R\ref{R1})--(R\ref{R9}).
\end{proof}

The above result implies that to prove Theorem~\ref{thm:main} for $\eps$-extremal graphs it will suffice to consider only graphs which are $ST$-extremal, $AB$-extremal or $ABST$-extremal. Indeed, to see that we may assume that $a\leq b$ and $s\leq t$, suppose that $G$ is $\eps$-extremal. Then $G$ has a partition satisfying (P\ref{P1})--(P\ref{P7}), (Q\ref{Q1})--(Q\ref{Q8}) or (R\ref{R1})--(R\ref{R9}) by Proposition~\ref{prop:structure2}. Note that relabelling the sets of the partition $(A,B,S,T)$ by $(B,A,T,S)$ if necessary allows us to assume that $a\leq b$. If $s\leq t$, then we are done. If $s>t$, reverse the orientation of every edge in $G$ to obtain the new graph $G'$. Relabel the sets $(A,B,S,T)$ by $(A,B,T,S)$. Under this new labelling, the graph $G'$ satisfies all of the original properties as well as $a\leq b$ and $s\leq t$. Obtain $C'$ from the cycle $C$ by reversing the orientation of every edge in $C$. The problem of finding a copy of $C$ in $G$ is equivalent to finding a copy of $C'$ in $G'$.


\section{$G$ is $ST$-extremal}\label{sec:ST}

The aim of this section is to prove the following lemma which settles Theorem~\ref{thm:main} in the case when $G$ is $ST$-extremal.

\begin{lemma}\label{lem:ST}
Suppose that $1/n \ll \eps_3 \ll \eps_4 \ll \eta_2 \ll 1.$
Let $G$ be a digraph on $n$ vertices such that $\delta^0(G)\geq n/2$ and $G$ is $ST$-extremal. If $C$ is any orientation of a cycle on $n$ vertices, then $G$ contains a copy of $C$.
\end{lemma}

We will split the proof of Lemma~\ref{lem:ST} into two cases based on how close the cycle $C$ is to being consistently oriented. Recall that $\sigma(C)$ denotes the number of sink vertices in $C$. Observe that in any oriented cycle, the number of sink vertices is equal to the number of source vertices.

\subsection{$C$ has many sink vertices, $\sigma(C)\geq \eps_4 n$}

The rough strategy in this case is as follows. We would like to embed half of the cycle $C$ into $G[S]$ and half into $G[T]$, making use of the fact that these graphs are nearly complete. At this stage, we also suitably assign the  vertices in $A\cup B$ to $G[S]$ or $G[T]$. We will partition $C$ into two disjoint paths, $P_S$ and $P_T$, each containing at least $\sigma(C)/8$ sink vertices, which will be embedded into $G[S]$ and $G[T]$. The main challenge we will face is finding appropriate edges to connect the two halves of the embedding.

\begin{lemma}\label{lem:linking1}
Suppose that $1/n \ll \eps_3 \ll \eps_4 \ll \eta_2 \ll 1.$ Let $G$ be a digraph on $n$ vertices with $\delta^0(G)\geq n/2$. Suppose $A,B,S,T$ is a partition of $V(G)$ satisfying (P\ref{P*})--(P\ref{P7}). Let $C$ be an oriented cycle on $n$ vertices with $\sigma(C) \geq \eps_4n$. Then there exists a partition $S^*, T^*$ of the vertices of $G$ and internally disjoint paths $R_1, R_2, P_S, P_T$ such that $C=(P_SR_1P_TR_2)$ and the following hold:
\begin{enumerate}[\rm(i)]
	\item $S\subseteq S^*$ and $T\subseteq T^*$;
	\item $|P_T|=|T^*|$;
	\item $P_S$ and $P_T$ each contain at least $\eps_4n/8$ sink vertices;
	\item $|R_i|\leq 3$ and $G$ contains disjoint copies $R_i^G$ of $R_i$ such that $R_1^G$ is an $ST$-path, $R_2^G$ is a $TS$-path and all interior vertices of $R_i^G$ lie in $S^*$.
\end{enumerate}
\end{lemma}

In the proof of Lemma~\ref{lem:linking1} we will need the following proposition.

\begin{prop}\label{prop:2edges}
Suppose that $1/n \ll \eps_3 \ll \eps_4 \ll \eta \ll 1.$ Let $G$ be a digraph on $n$ vertices with $\delta^0(G) \geq n/2$. Suppose $A, B, S, T$ is a partition of $V(G)$ satisfying (P\ref{P*})--(P\ref{P7}).
\begin{enumerate}[\rm(i)]
	\item If $a=b \in \{0,1\}$ then there are two disjoint edges between $S$ and $T$ of any given direction.\label{prop:2edges1}
	\item If $A = \emptyset$ then there are two disjoint $TS$-edges.\label{prop:2edges2}
	\item If $a= 1$ and $b \geq 2$ then there are two disjoint $TS$-edges.\label{prop:2edges4}
	\item There are two disjoint edges in $E(S,T\cup A) \cup E(T, S\cup B)$.\label{prop:2edges6} 
\end{enumerate}
\end{prop}

\begin{proof}
Let 
$$S':=\{x\in S : N^+_{A}(x), N^-_{B}(x) = \emptyset\} \text{ and } T':=\{x\in T : N^+_{B}(x), N^-_{A}(x) = \emptyset\}.$$
First we prove (\ref{prop:2edges1}). If $a=b \in \{0,1\}$ then it follows from (P\ref{P6}), (P\ref{P7}) that $|S'|,|T'| \geq n/4$. Since $s\leq t$, it is either the case that $s \leq (n-1)/2-b$ or $s=t=n/2-b$. If $s \leq (n-1)/2-b$ choose any $x\neq y \in S'$. Both $x$ and $y$ have at least $\lceil n/2-((n-1)/2 -b-1+b)\rceil=2$ inneighbours and outneighbours in $T$, so we find the desired edges. Otherwise $s=t=n/2-b$ and each vertex in $S'$ must have at least one inneighbour and at least one outneighbour in $T$ and each vertex in $T'$ must have at least one inneighbour and at least one outneighbour in $S$. It is now easy to check that (\ref{prop:2edges1}) holds. Indeed, K\"onig's theorem gives the two required disjoint edges provided they have the same direction. Using this, it is also easy to find two edges in opposite directions.%
\COMMENT{If we want edges in opposite directions, K\"onig's theorem gives two disjoint $ST$-edges $e_1$ and $e_2$. Choose any $t_1 \in T'$ which is disjoint from $e_1, e_2$. Now $t_1$ must have an outneighbour $s_1\in S$ and $s_1$ can lie on at most one of the edges $e_1$ and $e_2$. Then, without loss of generality, $e_1, t_1s_1$ form the required pair of edges.}

We now prove (\ref{prop:2edges2}). Suppose that $A = \emptyset$. We have already seen that the result holds when $B=\emptyset$. So assume that $b \geq 1$. Since $s \leq (n-b)/2$, each vertex in $S$ must have at least $b/2+1$ inneighbours in $T \cup B$. Assume for contradiction that there are no two disjoint $TS$-edges. Then all but at most one vertex in $S$ must have at least $b/2$ inneighbours in $B$. So $e(B, S) \geq bn/8$ which implies that there is a vertex $v \in B$ with $d^+_S(v) \geq n/8.$ But this contradicts (P\ref{P7}). So there must be two disjoint $TS$-edges.

For (\ref{prop:2edges4}), suppose that $a=1$ and $b \geq 2$. Since $s\leq (n-b-1)/2$, each vertex in $S$ must have at least $(b+1)/2$ inneighbours in $T\cup B$. Assume that there are no two disjoint $TS$-edges. Then all but at most one vertex in $S$ have at least $(b-1)/2$ inneighbours in $B$. So $e(B,S) \geq nb/12$ which implies that there is a vertex $v \in B$ with $d^+_S(v) \geq n/12$ which contradicts (P\ref{P7}). Hence (\ref{prop:2edges4}) holds.

For (\ref{prop:2edges6}), we observe that $\min\{s+b, t+a\} \leq (n-1)/2$ or $s+b=t+a=n/2$. If $s+b\leq (n-1)/2$ then each vertex in $S$ has at least two outneighbours in $T\cup A$, giving the desired edges. A similar argument works if $t+a \leq (n-1)/2$. If $s+b=t+a=n/2$ then each vertex in $S$ has at least one outneighbour in $T \cup A$ and each vertex in $T$ has at least one outneighbour in $S\cup B$. It is easy to see that there must be two disjoint edges in $E(S,T\cup A) \cup E(T, S\cup B)$.
\end{proof}

\begin{proofof}\textbf{Lemma~\ref{lem:linking1}.}
Observe that $C$ must have a subpath $P_1$ of length $n/3$ containing at least $\eps_4 n/3$ sink vertices. Let $v\in P_1$ be a sink vertex such that the subpaths $(P_1v)$ and $(vP_1)$ of $P_1$ each contain at least $\eps_4 n/7$ sink vertices. Write $C=(v_1v_2\dots v_n)$ where $v_1:=v$ and write $k':=n-t$.

\medskip

\noindent \textbf{Case 1: } \emph{$a\leq 1$}

If $a=b$, set $S^*:=S\cup A\cup B$, $T^*:=T$, $R_1:=(v_{k'}v_{k'+1})$ and $R_2:=(v_nv_1)=v_nv_1$. By Proposition~\ref{prop:2edges}(\ref{prop:2edges1}), $G$ contains a pair of disjoint edges between $S$ and $T$ of any given orientation. So we can map $v_nv_1$ to a $TS$-edge and $(v_{k'}v_{k'+1})$ to an edge between $S$ and $T$ of the correct orientation such that the two edges are disjoint.

Suppose now that $b \geq a+1$. By Proposition~\ref{prop:2edges}(\ref{prop:2edges2})--(\ref{prop:2edges4}), we can find two disjoint $TS$-edges $e_1$ and $e_2$. If $v_{k'}$ is not a source vertex, set $S^*:=S\cup A\cup B$, $T^*:=T$, $R_1:=(v_{k'-1}v_{k'}v_{k'+1})$ and $R_2:=v_nv_1$.  Map $v_nv_1$ to $e_1$.  If $v_{k'+1}v_{k'}\in E(C)$, map $R_1$ to a path of the form $SST$ which uses $e_2$. Otherwise, since $v_{k'}$ is not a source vertex, $R_1$ is a forward path. Using (P\ref{P7}), we find a forward path of the form $SBT$ for $R_1^G$.

So let us suppose that $v_{k'}$ is a source vertex. Let $b_1\in B$ and set $S^*:=S\cup A\cup B \setminus\{b_1\}$ and $T^*:=T\cup\{b_1\}$. Let $R_1:=(v_{k'-1}v_{k'})=v_{k'}v_{k'-1}$ and $R_2:=v_nv_1$. We know that $v_nv_1, v_{k'}v_{k'-1}\in E(C)$, so we can map these edges to $e_1$ and $e_2$.

In each of the above, we define $P_S$ and $P_T$ to be the paths, which are internally disjoint from $R_1$ and $R_2$, such that $C=(P_SR_1P_TR_2)$. Note that (i)--(iv) are satisfied.

\medskip

\noindent \textbf{Case 2: } \emph{$a\geq 2$}

Apply Proposition~\ref{prop:2edges}(\ref{prop:2edges6}) to find two disjoint edges $e_1, e_2 \in E(S,T\cup A) \cup E(T, S\cup B)$. Choose any distinct $x,y\in A\cup B$ such that $x$ and $y$ are disjoint from $e_1$ and $e_2$.

First let us suppose that $v_{k'}$ is a sink vertex. If $e_1, e_2 \in E(S,A)\cup E(T, S \cup B)$, set $S^*:=S\cup A\cup B$, $T^*:=T$, $R_1:=(v_{k'-1}v_{k'}v_{k'+1})$ and $R_2:=(v_nv_1v_2)$. If $e_1\in E(T, S \cup B)$, use (P\ref{P2}) and (P\ref{P7}) to find a path of the form $S(S\cup B)T$ which uses $e_1$ for $R_1^G$. If $e_1 \in E(S,A)$, we use (P\ref{P6}) to find a path of the form $SAT$ using $e_1$ for $R_1^G$. In the same way, we find a copy $R_2^G$ of $R_2$. If exactly one of $e_i$, $e_2$ say, lies in $E(S,T)$, set $S^*:=(S\cup A\cup B)\setminus\{x\}$, $T^*:=T\cup \{x\}$, $R_1:=(v_{k'-1}v_{k'}v_{k'+1})$ and $R_2:=(v_1v_2)$. Then $v_2v_1$ can be mapped to $e_2$ and we use $e_1$ to find a copy $R_1^G$ of $R_1$ as before. If both $e_1,e_2 \in E(S,T)$, set $S^*:=(S\cup A\cup B)\setminus\{x,y\}$, $T^*:=T\cup \{x,y\}$, $R_1:=(v_{k'-1}v_{k'})$ and $R_2:=(v_1v_2)$. Then map $v_2v_1$ and $v_{k'-1}v_{k'}$ to the edges $e_1$ and $e_2$.

Suppose now that $(v_{k'-1}v_{k'}v_{k'+1})$ is a consistently oriented path. If $e_2 \not \in E(S,T)$, let $S^*:=S\cup A\cup B$, $T^*:=T$, $R_1:=(v_{k'-1}v_{k'}v_{k'+1})$ and $R_2:=(v_nv_1v_2)$ and, if $e_2\in E(S,T)$, let $S^*:=(S\cup A\cup B)\setminus\{x\}$, $T^*:=T\cup \{x\}$, $R_1:=(v_{k'-1}v_{k'}v_{k'+1})$ and $R_2:=(v_1v_2)$. Then use the edge $e_2$ to find a copy $R_2^G$ of $R_2$ as above. We use (P\ref{P6}) or (P\ref{P7}) to map $R_1$ to a backward path of the form $SAT$ or a forward path of the form $SBT$ as appropriate.

We let $P_S$ and $P_T$ be paths which are internally disjoint from $R_1$ and $R_2$ such that $C=(P_SR_1P_TR_2)$. Then (i)--(iv) are satisfied.

It remains to consider the case when $v_{k'}$ is a source vertex. We now consider the vertex $v_{k'-1}$ instead of $v_{k'}$. Note that $C$ cannot contain two adjacent source vertices, so either $v_{k'-1}$ is a sink vertex or $(v_{k'-2}v_{k'-1}v_{k'})$ is a backward path. We proceed as previously. Note that when we define the path $P_T$ it will have one additional vertex and so we must allocate an additional vertex from $A\cup B$ to $T^*$, we are able to do this since $a+b>3$.%
\COMMENT{Suppose that $v_{k'-1}$ is a sink vertex. If $e_1, e_2 \in E(S,A)\cup E(T, S \cup B)$, set $S^*:=(S\cup A\cup B)\setminus\{x\}$, $T^*:=T\cup \{x\}$, $R_1:=(v_{k'-2}v_{k'-1}v_{k'})$ and $R_2:=(v_nv_1v_2)$. Find $R_1^G$ and $R_2^G$ as before. If exactly one $e_i$, $e_2$ say, lies in $E(S,T)$, we set $S^*:=S\cup A\cup B\setminus\{x,y\}$, $T^*:=T\cup \{x,y\}$, $R_1:=(v_{k'-2}v_{k'-1}v_{k'})$ and $R_2:=(v_1v_2)$. Then $v_2v_1$ can be mapped to $e_2$ and we use $e_1$ to find a copy of $R_1^G$ as before. If both $e_1,e_2 \in E(S,T)$, let $z\in A\cup B$, $z\neq x,y$ be arbitrary. Set $S^*:=S\cup A\cup B\setminus\{x,y,z\}$, $T^*:=T\cup \{x,y,z\}$, $R_1:=(v_{k'-2}v_{k'-1})$ and $R_2:=(v_1v_2)$. Then map $v_2v_1$ and $v_{k'-2}v_{k'-1}$ to the edges $e_1$ and $e_2$.\\
Suppose now that $(v_{k'-2}v_{k'-1}v_{k'})$ is a backward path. If $e_2 \not \in E(S,T)$, let $S^*:=(S\cup A\cup B)\setminus\{x\}$, $T^*:=T\cup \{x\}$, $R_1:=(v_{k'-2}v_{k'-1}v_{k'})$ and $R_2:=(v_nv_1v_2)$ and, if $e_2\in E(S,T)$, let $S^*:=S\cup A\cup B\setminus\{x,y\}$, $T^*:=T\cup \{x,y\}$, $R_1:=(v_{k'-2}v_{k'-1}v_{k'})$ and $R_2:=(v_1v_2)$. Then use the edge $e_2$ to find a copy of $R_2^G$ as above. We use (P\ref{P7}) to map $R_1$ to a backward path of the form $SAT$. Let $P_S$ and $P_T$ be such that $C=(P_SR_1P_TR_2)$.}
\end{proofof}

Apply Lemma~\ref{lem:linking1} to $G$ and $C$ to obtain internally disjoint subpaths $R_1$, $R_2$, $P_S$ and $P_T$ of $C$ as well as a partition $S^*,T^*$ of $V(G)$. Let $R_i^G$ be copies of $R_i$ in $G$ satisfying the properties of the lemma. Write $R'$ for the set of interior vertices of the $R_i^G$. Define $G_S:= G[S^*\setminus R']$ and $G_T:=G[T^*]$. Let $x_T$ and $x_S$ be the images of the final vertices of $R_1$ and $R_2$ and let $y_S$ and $y_T$ be the images of the initial vertices of $R_1$ and $R_2$, respectively. Also, let
$V_S:=S^* \cap (A \cup B)$ and $V_T:=T^* \cap (A \cup B)$.

The following proposition allows us to embed copies of $P_S$ and $P_T$ in $G_S$ and $G_T$. The idea is to greedily find a short path which will contain all of the vertices in $V_S$ and $V_T$ and any vertices of ``low degree''. We then use that the remaining graph is nearly complete to complete the embedding.

\begin{prop}\label{prop:pathembedding}
Let $G_S$, $P_S$, $P_T$, $x_S$, $y_S$, $x_T$ and $y_T$ be as defined above. 
\begin{enumerate}[\rm(i)]
	\item There is a copy of $P_S$ in $G_S$ such that the initial vertex of $P_S$ is mapped to $x_S$ and the final vertex is mapped to $y_S$.
	\item There is a copy of $P_T$ in $G_T$ such that the initial vertex of $P_T$ is mapped to $x_T$ and the final vertex is mapped to $y_T$.
\end{enumerate}
\end{prop}

\begin{proof}
We prove (i), the proof of (ii) is identical. Write $P_S=(u_1u_2\dots u_k)$. An averaging argument shows that there exists a subpath $P$ of $P_S$ of order at most $\eps_4 n$ containing at least $\sqrt{\eps_3}n$ sink vertices.%
\COMMENT{Consider a partition of $P_S$ into subpaths of order $\eps_4 n$. If there does not exist such $P$ then the number of sink vertices in $P_S$ is at most $|P_S|(\sqrt{\eps_3}n+1)/(\eps_4n) < \eps_4n/8$, a contradiction.}

Let 
$X:= \{x\in S : d^+_S(x)< n/2-\eps_3 n \text{ or } d^-_S(x) < n/2-\eps_3 n\}$. By (P\ref{P3}), $|X| \leq \eps_3n$ and so, using (P\ref{P2}), we see that every vertex $x\in X$ is adjacent to at least $\eta_2n/2$ vertices in $S\setminus X$. So we can assume that $x_S, y_S \in S \setminus X$ since otherwise we can embed the second and penultimate vertices on $P_S$ to vertices in $S\setminus X$ and consider these vertices instead.

Let $u_1'$ be the initial vertex of $P$ and $u_k'$ be the final vertex. Define $m_1 := d_{P_S}(u_1,u_1')+1$ and $m_2:=d_{P_S}(u_k',u_k)+1$. Suppose first that $m_1,m_2 > \eta_2^2n$. We greedily find a copy $P^G$ of $P$ in $G_S$ which covers all vertices in $V_S\cup X$ such that $u_1'$ and $u_k'$ are mapped to vertices $s_1,s_2\in S \setminus X$. This is possible since any two vertices in $X$ can be joined by a path of length at most three of any given orientation, by (P\ref{P2}) and (P\ref{P3}), and we can use each vertex in $V_S$ as the image of a sink or source vertex of $P$.
Partition $(V(G_S)\setminus V(P^G)) \cup \{s_1,s_2\}$, arbitrarily, into two sets $L_1$ and $L_2$ of size $m_1$ and $m_2$ respectively so that $s_1,x_S \in L_1$ and $s_2,y_S \in L_2$. Consider the graphs $G_i:=G_S[L_i]$ for $i=1,2$. Then (P\ref{P3}) implies that
$\delta(G_i) \geq m_i-\eps_3 n -\eps_4n \geq 7m_i/8.$
Applying Proposition~\ref{prop:completepath}(i), we find suitably oriented Hamilton paths from $s_1$ to $x_S$ in $G_1$ and $s_2$ to $y_S$ in $G_2$ which, when combined with $P$, form a copy of $P_S$ in $G_S$ (with endvertices $x_S$ and $y_S$).

It remains to consider the case when $m_1< \eta_2^2n$ or $m_2 < \eta_2^2n$. Suppose that the former holds (the latter is similar). Let $P'$ be the subpath of $P_S$ between $u_1$ and $u_k'$. So $P \subseteq P'$. Similarly as before, we first greedily find a copy of $P'$ in $G_S$ which covers all vertices of $X\cup V_S$ and then extend this to an embedding of $P_S$.%
\COMMENT{Greedily embed into $G_S$ a path starting from $x_S$ which is isomorphic to $P'$ and incorporates all vertices in $V_S$, as sink or source vertices, and all vertices in $X$. We ensure that the final vertex of $P'$, $u_1'$, is mapped to a vertex $s_1\in S \setminus X$. Now $|P'| \leq 2\eta_2^2n$. Let $S':=(V(G_S) \setminus V((P')^G))\cup \{s_1\}$ and consider the graph $G_S':=G_S[S']$. Note that $\delta^0(G_S') \geq n/2-\eps_3n -2\eta_2^2n \geq 7|G_S'|/8$ so $G_S'$ has a spanning path of any given orientation from $s_1$ to $y_S$, by Proposition~\ref{prop:completepath}(i), and we obtain an embedding of $P_S$ in $G_S$.}
\end{proof}

Proposition~\ref{prop:pathembedding} allows us to find copies of $P_S$ and $P_T$ in $G_S$ and $G_T$ with the desired endvertices. Combining these with $R_1^G$ and $R_2^G$ found in Lemma~\ref{lem:linking1}, we obtain a copy of $C$ in $G$. This proves Lemma~\ref{lem:ST} when $\sigma(C) \geq \eps_4n$.

\subsection{$C$ has few sink vertices, $\sigma(C)<\eps_4n$}

Our approach will closely follow the argument when $C$ had many sink vertices. The main difference will be how we cover the exceptional vertices, i.e. the vertices in $A\cup B$. We will call a consistently oriented subpath of $C$ which has length $20$ a \emph{long run}. If $C$ contains few sink vertices, it must contain many of these long runs. So, whereas previously we used sink and source vertices, we will now use long runs to cover the vertices in $A\cup B$.

\begin{prop}\label{prop:goodpaths}
Suppose that $1/n \ll \eps \ll 1$ and $n/4 \leq k\leq 3n/4$. Let $C$ be an oriented cycle with $\sigma(C) < \eps n$. Then we can write $C$ as $(u_1u_2 \dots u_n)$ such that there exist:
\begin{enumerate}[\rm(i)]
\item Long runs $P_1, P_2$ such that $P_1$ is a forward path and $d_C(P_1, P_2)=k$,
\item Long runs $P_1', P_2', P_3', P_4'$ such that $d_C(P_i', P_{i+1}')=\lfloor n/4 \rfloor$ for $i=1,2,3$.
\end{enumerate}
\end{prop}

\begin{proof}
Let $P$ be a subpath of $C$ of length $n/8$. Let $\cQ$ be a consistent collection of vertex disjoint long runs in $P$ of maximum size. Then $|\cQ| \geq 2\eps n$, with room to spare. We can write $C$ as $(u_1u_2 \dots u_n)$ so that the long runs in $\cQ$ are forward paths.

Suppose that (i) does not hold. For each $Q_i \in \cQ$, let $Q_i'$ be the path of length $20$ such that $d_C(Q_i,Q_i')=k$. Since $Q_i'$ is not a long run, $Q_i'$ must contain at least one sink or source vertex. The paths $Q_i'$ are disjoint so, in total, $C$ must contain at least
$|\cQ|/2 \geq \eps n > \sigma(C)$ sink vertices, a contradiction. Hence (i) holds.

We call a collection of four disjoint long runs $P_1, P_2, P_3, P_4$ \emph{good} if $P_1 \in \cQ$ and $d_C(P_i,P_{i+1}) = \lfloor n/4 \rfloor$ for all $i=1, 2, 3$. Suppose $C$ does not contain a good collection of long runs. In particular, this means that each long run in $\cQ$ does not lie in a good collection. For each path $Q_i \in \cQ$, let $Q_{i,1}, Q_{i,2}, Q_{i,3}$ be subpaths of $C$ of length $20$ such that $d_C(Q_i, Q_{i,j})=j\lfloor n/4 \rfloor$. Since $\{Q_i, Q_{i,1}, Q_{i,2}, Q_{i,3}\}$ does not form a good collection, at least one of the $Q_{i,j}$ must contain a sink or source vertex. The paths $Q_{i,j}$ where $Q_i \in \cQ$ and $j=1,2,3$ are disjoint so, in total, $C$ must contain at least
$|\cQ|/2 \geq \eps n>\sigma(C)$ sink vertices, which is a contradiction. This proves (ii).
\end{proof}

The following proposition finds a collection of edges oriented in an atypical direction for an $\eps$-extremal graph. We will use these edges to find consistently oriented $S$- and $T$-paths covering all of the vertices in $A\cup B$. This proposition will be used again in Section~\ref{sec:ABST1}, where it allows us to correct an imbalance in the sizes of $A$ and $B$.

\begin{prop}\label{prop:dedges}
Let $G$ be a digraph on $n$ vertices with $\delta^0(G) \geq n/2$. Let $d\geq 0$ and suppose $A, B, S, T$ is a partition of $V(G)$ into sets of size $a,b,s,t$ with $t\geq s\geq d+2$ and $b = a+d$. Then $G$ contains a collection $M$ of $d+1$ edges in $E(T,S \cup B) \cup E(B, S)$ satisfying the following. The endvertices of $M$ outside $B$ are distinct and each vertex in $B$ is the endvertex of at most one $TB$-edge and at most one $BS$-edge in $M$. Moreover, if $e(T,S)>0$, then $M$ contains a $TS$-edge.
\end{prop}

\begin{proof}
Let $k:=t-s$.
We define a bipartite graph $G'$ with vertex classes $S':=S\cup B$ and $T':=T\cup B$ together with all edges $xy$ such that $x \in S', y\in T'$ and $yx\in E(T,S \cup B) \cup E(B, S)$. We claim that $G'$ has a matching of size $d+2$. 
To prove the claim, suppose that $G'$ has a vertex cover $X$ of size $|X|< d+2$. Then $|X\cap S'|<(d-k)/2+1$ or $|X\cap T'|<(d+k)/2+1$. 
Suppose that the former holds and consider any vertex $t_1 \in T \setminus X$. Since $\delta^+(G) \geq n/2$ and $a+t=(n-d+k)/2$, $t_1$ has at least $(d-k)/2+1$ outneighbours in $S'$. But these vertices cannot all be covered by $X$. So we must have that $|X\cap T'|<(d+k)/2+1$. Consider any vertex $s_1 \in S \setminus X$. Now $\delta^-(G)\geq n/2$ and $a+s=(n-d-k)/2$, so $s_1$ must have at least $(d+k)/2+1$ inneighbours in $T'$. But not all of these vertices can be covered by $X$. Hence, any vertex cover of $G'$ must have size at least $d+2$ and so K\"onig's theorem implies that $G'$ has a matching of size $d+2$.%
\COMMENT{Argument is fine if $(d-k)/2+1 \leq 0$. Then we must have that $|X\cap T'|<d+2\leq (d+k)/2+1$ and argument follows through.}

If $e(T,S)>0$, either the matching contains a $TS$-edge, or we can choose any $TS$-edge $e$ and at least $d$ of the edges in the matching will be disjoint from $e$. This corresponds to a set of $d+1$ edges in $E(T,S \cup B) \cup E(B, S)$ in $G$ with the required properties.
\end{proof}

We define a \emph{good path system} $\cP$ to be a collection of disjoint $S$- and $T$-paths such that each path $P \in \cP$ is consistently oriented, has length at most six and covers at least one vertex in $A\cup B$. Each good path system $\cP$ gives rise to a modified partition $A_{\cP}, B_{\cP}, S_{\cP}, T_{\cP}$  of the vertices of $G$ (we allow $A_{\cP}$, $B_{\cP}$ to be empty) as follows. Let $\textrm{Int}_S(\cP)$ be the set of all interior vertices on the $S$-paths in $\cP$ and $\textrm{Int}_T(\cP)$ be the set of all interior vertices on the $T$-paths. We set $A_{\cP}:=A \setminus V(\cP)$, $B_{\cP}:=B \setminus V(\cP)$, $S_{\cP}:=(S\cup \textrm{Int}_S(\cP))\setminus\textrm{Int}_T(\cP)$ and $T_{\cP}:=(T\cup \textrm{Int}_T(\cP)) \setminus \textrm{Int}_S(\cP)$ and say that $A_{\cP}, B_{\cP}, S_{\cP}, T_{\cP}$ is the \emph{$\cP$-partition} of $V(G)$.

\begin{lemma}\label{lem:linking2}
Suppose that $1/n \ll \eps_3 \ll \eps_4 \ll \eta_2 \ll 1.$ Let $G$ be a digraph on $n$ vertices with $\delta^0(G) \geq n/2$. Suppose $A, B, S, T$ is a partition of $V(G)$ satisfying (P\ref{P*})--(P\ref{P7}). Let $C$ be a cycle on $n$ vertices with $\sigma(C) < \eps_4n$. 
Then there exists $t^*$ such that one of the following holds:
\begin{itemize}
	\item There exist internally disjoint paths $P_S, P_T, R_1, R_2$ such that:
	\begin{enumerate}[\rm(i)]
		\item $C=(P_SR_1P_TR_2)$;
		\item $|P_T|=t^*$;
		\item $R_1$ and $R_2$ are paths of length two and $G$ contains disjoint copies $R_i^G$ of $R_i$ whose interior vertices lie in $V(G) \setminus T$. Moreover, $R_1^G$ is an $ST$-path and $R_2^G$ is a $TS$-path.
	\end{enumerate}
	\item There exist internally disjoint paths $P_S, P_S', P_T, P_T', R_1, R_2, R_3, R_4$ such that: 
	\begin{enumerate}[\rm(i)]
		\item $C=(P_SR_1P_TR_2P_S'R_3P_T'R_4)$;
		\item $|P_T|+|P_T'|=t^*$ and $|P_S|, |P_S'|, |P_T|, |P_T'| \geq n/8$;
		\item $R_1,R_2,R_3, R_4$ are paths of length two and $G$ contains disjoint copies $R_i^G$ of $R_i$ whose interior vertices lie in $V(G) \setminus T$. Moreover, $R_1^G$ and $R_3^G$ are $ST$-paths and $R_2^G$ and $R_4^G$ are $TS$-paths.			
	\end{enumerate}
\end{itemize}
Furthermore, $G$ has a good path system $\cP$ such that the paths in $\cP$ are disjoint from each $R_i^G$, $\cP$ covers $(A\cup B)\setminus \bigcup V(R_i^G)$ and the $\cP$-partition $A_{\cP}, B_{\cP}, S_{\cP}, T_{\cP}$  of $V(G)$ satisfies $|T_{\cP}|=t^*$. 
\end{lemma}

\begin{proof}
Let $d:=b-a$ and $k:=t-s$.

We first obtain a good path system $\cP_0$ covering $A\cup B$ as follows. Apply Proposition~\ref{prop:dedges} to obtain a collection $M_0$ of $d+1$ edges as described in the proposition. Choose $M\subseteq M_0$ of size $d$ such that $M$ contains a $TS$-edge if $d\geq 1$ and $e(T,S)>0$. We use each edge $e\in M$ together with properties (P\ref{P2}), (P\ref{P4}) and (P\ref{P7}) to cover one vertex in $B$ by a consistently oriented path of length at most six as follows. If $e\in E(T,B)$ and $e$ is disjoint from all other edges in $M$, find a consistently oriented path of the form $TBT$ using $e$. If $e\in E(B,S)$ and $e$ is disjoint from all other edges in $M$, find a consistently oriented path of the form $SBS$ using $e$. If $e\in E(T,S)$, we note that (P\ref{P2}), (P\ref{P4}) and (P\ref{P7}) allows us to find a consistently oriented path of length three between any vertex in $B$ and any vertex in $T$. So we can find a consistently oriented path of the form $SB(T)^3S$ which uses $e$. Finally, if $e\in E(T,B)$ and shares an endvertex with another edge $e'\in M\cap E(B,S)$ we find a consistently oriented path of the form $SB(T)^3BS$ using $e$ and $e'$. This path uses two edges in $M$ but covers two vertices in $B$. Since we have many choices for each such path, we can choose them to be disjoint, so $M$ allows us to find a good path system $\cP_1$ covering $d$ vertices in $B$.

Label the vertices in $A$ by $a_1, a_2, \dots, a_a$ and the remaining vertices in $B$ by $b_1, b_2, \dots, b_a$. We now use (P\ref{P5})--(P\ref{P7}) to find a consistently oriented $S$- or $T$-path $L_i$ covering each pair $a_i, b_i$. If $1\leq i \leq \lceil(4a+k)/8\rceil$, cover the pair $a_i,b_i$ by a path of the form $SBTAS$. If $\lceil(4a+k)/8\rceil<i \leq a$ cover the pair $a_i,b_i$ by a path of the form $TASBT$.%
\COMMENT{bound $\lceil(4a+k)/8\rceil$ used in Case~2.3.} 
Let $\cP_2:=\bigcup_{i=1}^a L_i$.

We are able to choose all of these paths so that they are disjoint and thus obtain a good path system $\cP_0:=\cP_1\cup \cP_2$ covering $A\cup B$. Let $A_{\cP_0}, B_{\cP_0}, S_{\cP_0}, T_{\cP_0}$ be the $\cP_0$-partition of $V(G)$ and let $t':=|T_{\cP_0}|$, $s':=|S_{\cP_0}|$.

By Proposition~{\ref{prop:goodpaths}}(i), we can enumerate the vertices of $C$ so that there are long runs $P_1, P_2$ such that $P_1$ is a forward path and $d_C(P_1, P_2) = t'$. We will find consistently oriented $ST$- and $TS$-paths for $R_1^G$ and $R_2^G$ which depend on the orientation of $P_2$. The paths $R_1$ and $R_2$ will be consistently oriented subpaths of $P_1$ and $P_2$ respectively, whose position will be chosen later.

\medskip

\noindent \textbf{Case 1: } \emph{$b\geq a+2$.}

Suppose first that $P_2$ is a backward path. If $\cP_1$ contains a path of the form $SB(T)^3BS$, let $b_0$ and $b_0'$ be the two vertices in $B$ on this path. Otherwise, let $b_0$ and $b_0'$ be arbitrary vertices in $B$ which are covered by $\cP_1$. Use (P\ref{P7}) to find a forward path for $R_1^G$ which is of the form $S\{b_0\}T$. We also find a backward path of the form $T\{b_0'\}S$ for $R_2^G$. We choose the paths $R_1^G$ and $R_2^G$ to be disjoint from all paths in $\cP_0$ which do not contain $b_0$ or $b_0'$.

Suppose now that $P_2$ is a forward path. If $a\geq 1$, consider the path $L_1\in \cP_2$ covering $a_1\in A$ and $b_1\in B$. Find forward paths of the form $S\{b_1\}T$ for $R_1^G$ and $T\{a_1\}S$ for $R_2^G$, using (P\ref{P6}) and (P\ref{P7}), which are disjoint from all paths in $\cP_0\setminus\{L_1\}$. Finally, we consider the case when $a=0$. Recall that $e(T,S)>0$ by Proposition~\ref{prop:2edges}(ii) and so $M$ contains a $TS$-edge. Hence there is a path $P'$ in $\cP_1$ of the form $SB(T)^3S$, covering a vertex $b_0 \in B$ and an edge $t_1s_1\in E(T,S)$, say. We use (P\ref{P2}) and (P\ref{P7}) to find forward paths of the form $S\{b_0\}T$ for $R_1^G$ and $\{t_1\}\{s_1\}S$ for $R_2^G$ which are disjoint from all paths in $\cP_0\setminus\{P'\}$.

Obtain the good path system $\cP$ from $\cP_0$ by removing all paths meeting $R_1^G$ or $R_2^G$. Let $A_{\cP}, B_{\cP}, S_{\cP}, T_{\cP}$ be the $\cP$-partition of $V(G)$ and $t^*:=|T_{\cP}|$. The only vertices which could have moved to obtain $T_{\cP}$ from $T_{\cP_0}$ are interior vertices on the paths in $\cP_0\setminus \cP$, so $|t^*-t'| \leq 2\cdot 5=10$. Thus we can choose $R_1$ and $R_2$ to be subpaths of length two of $P_1$ and $P_2$ so that $|P_T|=t^*$, where $P_S$ and $P_T$ are defined by $C=(P_SR_1P_TR_2)$.

\medskip

\noindent \textbf{Case 2: } \emph{$b\leq a+1$.}

\medskip

\noindent \textbf{Case 2.1: } \emph{$a\leq 1$.}

If $a=b$, by Proposition~\ref{prop:2edges}(\ref{prop:2edges1}) we can find disjoint $e_1, e_2\in E(S,T)$ and disjoint $e_3 \in E(S,T)$, $e_4 \in E(T,S)$. Note that $\cP_0=\cP_2$, since $a=b$, so we may assume that all paths in $\cP_0$ are disjoint from $e_1,e_2,e_3, e_4$. If $P_2$ is a forward path, find a forward path of the form $SST$ for $R_1^G$ using $e_3$ and a forward path of the form $TSS$ for $R_2^G$ using $e_4$. If $P_2$ is a backward path, find a forward path of the form $SST$ for $R_1^G$ using $e_1$ and a backward path of the form $TSS$ for $R_2^G$ using $e_2$. In both cases, we choose $R_1^G$ and $R_2^G$ to be disjoint from all paths in $\cP_0$.

If $b=a+1$, note that there exist $e_1\in E(S,T)$ and $e_2\in E(T,S)$. (To see this, use that $\delta^0(G) \geq n/2$ and the fact that (P\ref{P6}) and (P\ref{P7})  imply that $|\{x\in S: N^+_A(x), N^-_B(x)=\emptyset\}|\geq n/4$.)%
\COMMENT{We have $a=0, b=1$ or $a=1, b=2$. Let $S':=\{x\in S: N^+_A(x), N^-_B(x)=\emptyset\}$. By (P\ref{P6}) and (P\ref{P7}), $|S'|\geq n/4$. Note that $a+s \leq (n-1)/2$ so each vertex in $S'$ has at least two inneighbours in $T$. This gives disjoint $TS$-edges $e_1$ and $e_1'$. We also note that $b+s \leq (n+1)/2$ so each vertex in $S'$ must have at least one outneighbour in $T$. Choose a vertex $x$ in $S'$ which does not lie on either $e_1$ or $e_1'$ and let $e_2$ be an $ST$-edge using $x$. At least one of $e_1, e_1'$ must be disjoint from $e_2$.}
 We may assume that all paths in $\cP_2$ are disjoint from $e_1,e_2$. Let $b_0\in B$ be the vertex covered by the single path in $\cP_1$. Find a forward path of the form $S\{b_0\}T$ for $R_1^G$, using (P\ref{P7}). Find a consistently oriented path of the form $TSS$ for $R_2^G$ which uses $e_1$ if $P_2$ is a backward path and $e_2$ if $P_2$ is a forward path. Choose the paths $R_1^G$ and $R_2^G$ to be disjoint from the paths in $\cP_0\setminus \cP_1=\cP_2$.

In both cases, we obtain the good path system $\cP$ from $\cP_0$ by removing at most one path which meets $R_1^G$ or $R_2^G$. Let $A_{\cP}, B_{\cP}, S_{\cP}, T_{\cP}$ be the $\cP$-partition of $V(G)$ and let $t^*:=|T_{\cP}|$. The only vertices which could have moved to obtain $T_{\cP}$ from $T_{\cP_0}$ are interior vertices on the path in $\cP_0\setminus \cP$ if $\cP_0\neq \cP$, so $|t^*-t'| \leq 5$. So we can choose subpaths $R_i$ of $P_i$ so that $|P_T|=t^*$, where $P_S$ and $P_T$ are defined by $C=(P_SR_1P_TR_2)$.

\medskip

\noindent \textbf{Case 2.2: } \emph{$2\leq a \leq k$.}

If $P_2$ is a forward path, consider $a_1\in A$ and $b_1\in B$ which were covered by the path $L_1\in \cP_0$. Use (P\ref{P6}) and (P\ref{P7}) to find forward paths, disjoint from all paths in $\cP_0\setminus\{L_1\}$, of the form $S\{b_1\}T$ and $T\{a_1\}S$ for $R_1^G$ and $R_2^G$ respectively.

Suppose now that $P_2$ is a backward path. We claim that $G$ contains $2-d$ disjoint $ST$-edges. Indeed, suppose not. Then $d^+_T(x) \leq 1-d$ for all but at most one vertex in $S$. Note that $b+s = (n-k+d)/2$, so $d^+_{A\cup T}(x)\geq (k-d)/2+1$ for all $x\in S$. So
$$e(S,A) \geq (s-1)((k-d)/2+1-(1-d))=(s-1)(k+d)/2 \geq nk/8\geq na/8.$$
Hence, there is a vertex $x \in A$ with $d^-_S(x) \geq n/8$, contradicting (P\ref{P6}). Let $E=\{e_i:1\leq i \leq 2-d\}$ be a set of $2-d$ disjoint $ST$-edges. We may assume that $\cP_2$ is disjoint from $E$.

If $a=b$, use (P\ref{P2}) to find a forward path of the form $SST$ using $e_1$ for $R_1^G$ and a backward path of the form $TSS$ using $e_2$ for $R_2$. If $b=a+1$, let $b_0\in B$ be the vertex covered by the single path in $\cP_1$. Use (P\ref{P2}) and (P\ref{P7}) to find a forward path of the form $S\{b_0\}T$ for $R_1^G$ and a backward path of the form $TSS$ using $e_1$ for $R_2^G$. We choose the paths $R_1^G$ and $R_2^G$ to be disjoint from all paths in $\cP_2$.

In both cases, we obtain the good path system $\cP$ from $\cP_0$ by removing at most one path which meets $R_1^G$ or $R_2^G$. Let $A_{\cP}, B_{\cP}, S_{\cP}, T_{\cP}$ be the $\cP$-partition of $V(G)$ and $t^*:=|T_{\cP}|$. The only vertices which could have moved to obtain $T_{\cP}$ from $T_{\cP_0}$ are interior vertices on the path in $\cP_0\setminus \cP$ if $\cP_0\neq \cP$, so $|t^*-t'| \leq 5$. Thus we can choose $R_1$ and $R_2$ to be subpaths of length two of $P_1$ and $P_2$ so that $|P_T|=t^*$, where $P_S$ and $P_T$ are defined by $C=(P_SR_1P_TR_2)$.

\medskip

\noindent \textbf{Case 2.3: } \emph{$a \geq 2,k$.}

We note that
\begin{align*}
t'-s'&=|(T\cup \textrm{Int}_T(\cP_0))\setminus \textrm{Int}_S(\cP_0)|-|(S\cup \textrm{Int}_S(\cP_0))\setminus\textrm{Int}_T(\cP_0)|\\
&= |(T\cup \textrm{Int}_T(\cP_2))\setminus \textrm{Int}_S(\cP_2)|-|(S\cup \textrm{Int}_S(\cP_2))\setminus\textrm{Int}_T(\cP_2)|+c\\
&=(t+3a-4\lceil (4a+k)/8\rceil)-(s+4\lceil (4a+k)/8\rceil-a) +c\\
&=4a+k-8\lceil (4a+k)/8\rceil +c
\end{align*}
where $-7\leq c\leq 1$ is a constant representing the contribution of interior vertices on the path in $\cP_1$ if $b=a+1$ and $c=0$ if $b=a$.%
\COMMENT{$c=(|\textrm{Int}_T(\cP_1)\setminus T|-|\textrm{Int}_S(\cP_1) \cap T|)-(|\textrm{Int}_S(\cP_1)\setminus S|- |\textrm{Int}_T(\cP_1)\cap S|)$. If $\cP_1=\emptyset$, $c=0$. If $\cP_1$ consists of a path of form $TBT$, $c=1$. Path of form $SBS$, $c=-1$. Path of form $SB(T)^3S$, $c=-3-4=-7$.}
 In particular, this implies that $|t'-s'| \leq 15$%
\COMMENT{$-15=4a+k-8((4a+k)/8+1)-7 \leq t'-s'\leq 4a+k-8(4a+k)/8+1=1$} 
and
$$(n-15)/2 \leq s',t' \leq (n+15)/2.$$

Apply Proposition~\ref{prop:goodpaths}(ii) to find long runs $P'_1, P'_2, P'_3, P'_4$ such that $d_C(P'_i, P'_{i+1})=\lfloor n/4 \rfloor$ for $i=1,2,3$. Let $x_i$ be the initial vertex of each $P'_i$. If $\{P'_i, P'_{i+2}\}$ is consistent for some $i\in \{1,2\}$, consider $a_1 \in A$, $b_1 \in B$ which which were covered by the path $L_1\in \cP_0$. If $P'_i, P'_{i+2}$ are both forward paths, let $R_1^G$ and $R_2^G$ be forward paths of the form $S\{b_1\}T$ and $T\{a_1\}S$ respectively. If $P'_i, P'_{i+2}$ are both backward paths, let $R_1^G$ and $R_2^G$ be backward paths of the form $S\{a_1\}T$ and $T\{b_1\}S$ respectively. Choose the paths $R_1^G$ and $R_2^G$ to be disjoint from the paths in $\cP:=\cP_0\setminus \{L_1\}$. Let $A_{\cP}, B_{\cP}, S_{\cP}, T_{\cP}$ be the $\cP$-partition of $V(G)$ and let $t^*=|T_{\cP}|$. The only vertices which could have been added or removed to obtain $T_{\cP}$ from $T_{\cP_0}$ are interior vertices on $L_1$ so $(n-15)/2-3 \leq t^* \leq (n+15)/2+3$. Then we can choose $R_1$ and $R_2$ to be subpaths of length two of $P'_i$ and $P'_{i+2}$ so that $|P_T|=t^*$, where $P_S, P_T$ are defined so that $C=(P_SR_1P_TR_2)$.

So let us assume that $\{P'_i, P'_{i+2}\}$ is not consistent for $i=1,2$. We may assume that the paths $P'_1$ and $P'_4$ are both forward paths, by relabelling if necessary, and we illustrate the situation in Figure~\ref{fig:goodcoll}.

\begin{figure}[h]
\centering
\includegraphics[scale=0.33]{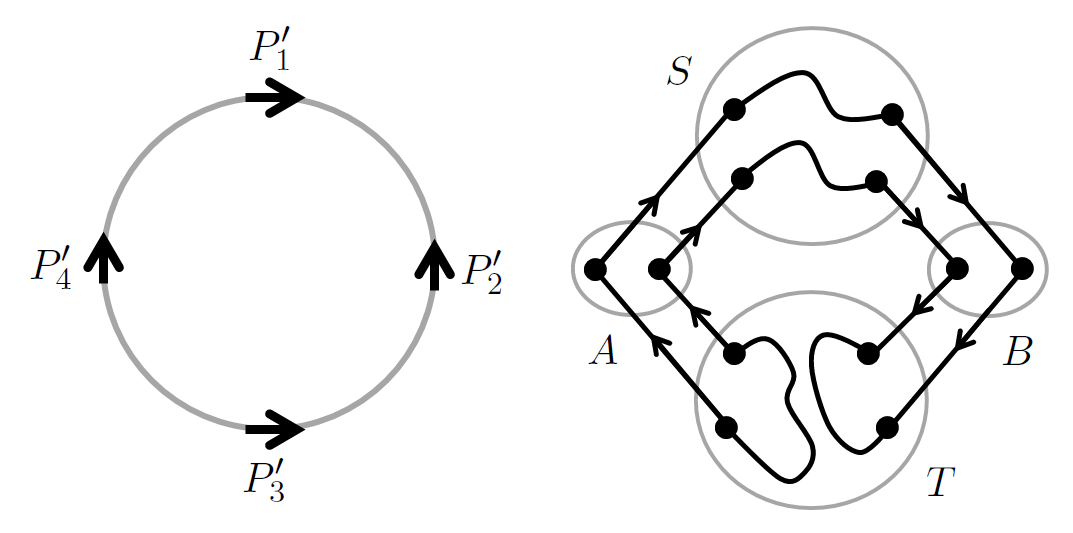}
\caption{A good collection of long runs.}\label{fig:goodcoll}
\end{figure}

Consider the vertices $a_i \in A$ and $b_i \in B$ covered by the paths $L_i\in\cP_0$ for $i=1,2$. Let $\cP:=\cP_0\setminus \{L_1, L_2\}$ and let $A_{\cP}, B_{\cP}, S_{\cP}, T_{\cP}$ be the $\cP$-partition of $V(G)$. Let $t^*:=|T_{\cP}|$. The only vertices which can have been added or removed to obtain $T_{\cP}$ from $T_{\cP_0}$ are interior vertices on the paths $L_1$ and $L_2$, so $(n-15)/2-6 \leq t^* \leq (n+15)/2+6$. Find a forward path of the form $S\{b_1\}T$ for $R_1^G$. Then find backward paths of the form $T\{b_2\}S$ and $S\{a_1\}T$ for $R_2^G$ and $R_3^G$ respectively. Finally, find a forward path of the form $T\{a_2\}S$ for $R_4^G$. We can choose the paths $R_i^G$ to be disjoint from all paths in $\cP$. Since $P'_1$ and $P'_2$ are of length $20$ we are able to find subpaths $R_1, R_2, R_3, R_4$ of $P'_1, P'_2, P'_3, P'_4$ so that $|P_T|+|P_T'|=t^*$, where $P_S, P_S', P_T, P_T'$ are defined so that $C=(P_SR_1P_TR_2P_S'R_3P_T'R_4)$.
\end{proof}

In order to prove Lemma~\ref{lem:ST} in the case when $\sigma(C)<\eps_4n$, we first apply Lemma~\ref{lem:linking2} to $G$. We now proceed similarly as in the case when $C$ has many sink vertices (see Proposition~\ref{prop:pathembedding}) and so we only provide a sketch of the argument. We first observe that any subpath of the cycle of length $100\eps_4n$ must contain at least
\begin{equation}\label{eqn:STlongruns}
\lfloor100\eps_4n/21\rfloor-2\eps_4n>2\eps_3n\geq a+b \geq |\cP|
\end{equation}
disjoint long runs.  Let $s_1$ be the image of the initial vertex of $P_S$. Let $P_S^*$ be the subpath of $P_S$ formed by the first $100\eps_4n$ edges of $P_S$. We can cover all $S$-paths in $\cP$ and all vertices $x\in S$ which satisfy $d^+_S(x)< n/2-\eps_3 n$ or $d^-_S(x)< n/2-\eps_3 n$ greedily by a path in $G$ starting from $s_1$ which is isomorphic to $P_S^*$. Note that \eqref{eqn:STlongruns} ensures that $P_S^*$ contains $|\cP|$ disjoint long runs. So we can map the $S$-paths in $\cP$ to subpaths of these long runs. Let $P_S''$ be the path formed by removing from $P_S$ all edges in $P_S^*$.

If Lemma~\ref{lem:linking2}(i) holds and thus $P_S$ is the only path to be embedded in $G[S]$, we apply Proposition~\ref{prop:completepath}(i) to find a copy of $P_S''$ in $G[S]$, with the desired endvertices. If Lemma~\ref{lem:linking2}(ii) holds, we must find copies of both $P_S$ and $P_S'$ in $G[S]$. So we split the graph into two subgraphs of the appropriate size before applying Proposition~\ref{prop:completepath}(i) to each.%
\COMMENT{Suppose (ii) (the proof for (i) is the same, without the initial subdivision step). Let $r_i$ be the endpoints of $R_i^G$ in $S$ for $i=1,2,3$ and let $r_4$ be the endpoint of the copy of $P_S^*$. Let $m_1:=|P_S''|$ and $m_2:=|P_S'|$. Arbitrarily partition the vertices $(S\setminus (V(P_S^*)) )\cup \{r_1, r_2, r_3, r_4\}$ into sets $M_1$ and $M_2$ so that $|M_1|=m_1$, $|M_2|=m_2$, $r_1,r_4\in M_1$ and $r_2,r_3\in M_2$. Consider the graphs $G_i:=G[M_i]$ for $i=1,2$. We have that $\delta^0(G_i)\geq m_i-\eps_3 n-100\eps_4n\geq 7m_i/8$, since $m_i \geq n/9$. Thus Proposition~\ref{prop:completepath}(i) implies that $G_1$ has a Hamilton path isomorphic to $P_S''$ from $r_4$ to $r_1$ and $G_2$ has a Hamilton path isomorphic to $P_S'$ from $r_2$ to $r_3$.}
 We do the same to find copies of $P_T$ (or $P_T$ and $P_T'$) in $G[T]$. Thus, we obtain a copy of $C$ in $G$.
This completes the proof of Lemma~\ref{lem:ST}.


\section{$G$ is $AB$-extremal}\label{sec:AB}

The aim of this section is to prove the following lemma which shows that Theorem~\ref{thm:main} is satisfied when $G$ is $AB$-extremal. Recall that an $AB$-extremal graph closely resembles a complete bipartite graph. We will proceed as follows. First we will find a short path which covers all of the exceptional vertices (the vertices in $S\cup T$). It is important that this path leaves a balanced number of vertices uncovered in $A$ and $B$. We will then apply Proposition~\ref{prop:completepath} to the remaining, almost complete, balanced bipartite graph to embed the remainder of the cycle.

\begin{lemma}\label{lem:AB}
Suppose that $1/n \ll \eps_3 \ll 1.$
Let $G$ be a digraph on $n$ vertices with $\delta^0(G)\geq n/2$ and assume that $G$ is $AB$-extremal. If $C$ is any orientation of a cycle on $n$ vertices which is not antidirected, then $G$ contains a copy of $C$.
\end{lemma}

If $b>a$, the next lemma implies that $E(B\cup T, B)$ contains a matching of size $b-a+2$. We can use $b-a$ of these edges to pass between vertices in $B$ whilst avoiding $A$ allowing us to correct the imbalance in the sizes of $A$ and $B$.

\begin{prop}\label{prop:balance}
Suppose $1/n \ll\eps_3 \ll 1$.
Let $G$ be a digraph on $n$ vertices with $\delta^0(G) \geq n/2$. Suppose $A, B, S, T$ is a partition of $V(G)$ satisfying (Q\ref{Q*})--(Q\ref{Q8}) and $b=a+d$ for some $d>0$. Then there is a matching of size $d+2$ in $E(B\cup T,B)$.
\end{prop}

\begin{proof}
Consider a maximal matching $M$ in $E(B \cup T,B)$ and suppose that $|M|\leq d+1$. Since $a+s \leq (n-d)/2$, each vertex in $B$ has at least $d/2$ inneighbours in $B\cup T$. In particular, since $M$ was maximal, each vertex in $B\setminus V(M)$ has at least $d/2$ inneighbours in $V(M)$. Then there is a $v \in V(M)\subseteq B\cup T$ with
$$d^+_B(v) \geq \frac{(b-2|M|)}{2|M|}\frac{d}{2}\geq \frac{n}{20},$$
contradicting (Q\ref{Q8}). Therefore $|M| \geq d+2$.
\end{proof}

We say that $P$ is an \emph{exceptional cover of $G$} if $P\subseteq G$ is a copy of a subpath of $C$ and
\begin{enumerate}[\rm{(EC}$1$)]
	\item $P$ covers $S\cup T$;\label{EC1}
	\item both endvertices of $P$ are in $A$;\label{EC2}
	\item $|A\setminus V(P)|+1=|B\setminus V(P)|$.\label{EC3}
\end{enumerate}

We will use the following notation when describing the form of a path. If $X,Y\in \{A,B\}$ then we write $X*Y$ for any path which alternates between $A$ and $B$ whose initial vertex lies in $X$ and final vertex lies in $Y$. For example, $A*A(ST)^2$ indicates any path of the form $ABAB\dots ASTST$.

Suppose that $P$ is of the form $Z_1Z_2\dots Z_m$, where $Z_i \in \{A,B,S,T\}$. Let $Z_{i_1}, Z_{i_2},\dots, Z_{i_j}$ be the appearances of $A$ and $B$, where $i_j<i_{j+1}$.  If $Z_{i_j}=A=Z_{i_{j+1}}$, we say that $Z_{i_{j+1}}$ is a \emph{repeated $A$}. We define a \emph{repeated $B$} similarly. Let $\text{rep}(A)$ and $\text{rep}(B)$ be the numbers of repeated $A$s and repeated $B$s, respectively. Suppose that $P$ has both endvertices in $A$ and $P$ uses $\ell+\text{rep}(B)$ vertices from $B$. Then $P$ will use $\ell+\text{rep}(A)+1$ vertices from $A$ (we add one because both endvertices of $P$ lie in $A$). So we have that
\begin{equation}\label{eqn:repeats}
|B\setminus V(P)|-|A\setminus V(P)|= b-a-\text{rep}(B)+\text{rep}(A)+1.
\end{equation}

Given a set of edges $M\subseteq E(G)$ we define the graph $G_M\subseteq G$ whose vertex set is $V(G)$ and whose edge set is $E(A,B\cup S)\cup E(B,A\cup T) \cup E(T,A) \cup E(S,B)\cup M\subseteq E(G)$. Informally, in addition to the edges of $M$, $G_M$ has edges between two vertex classes when the bipartite graph they induce in $G$ is dense.\label{G_M}

We will again split our argument into two cases depending on the number of sink vertices in~$C$.

\subsection{Finding an exceptional cover when $C$ has few sink vertices, $\sigma(C) <\eps_4n$}

It is relatively easy to find an exceptional cover when $C$ has few sink vertices by observing that $C$ must contain many disjoint consistently oriented paths of length three. We can use these consistently oriented paths to cover the vertices in $S\cup T$ by forward paths of the form $ASB$ or $BTA$, for example.

\begin{prop}\label{prop:excover1}
Suppose $1/n \ll \eps_3 \ll \eps_4 \ll 1$.
Let $G$ be a digraph on $n$ vertices with $\delta^0(G) \geq n/2$. Suppose $A, B, S, T$ is a partition of $V(G)$ satisfying (Q\ref{Q*})--(Q\ref{Q8}). If $\sigma(C) <\eps_4n$, then there is an exceptional cover of $G$ of length at most $21\eps_4n$.
\end{prop}

\begin{proof}
Let $d:=b-a$. Let $P$ be any subpath of $C$ of length $20\eps_4n$. Let $\cQ$ be a maximum consistent collection of disjoint paths of length three in $P$, such that $d_C(Q, Q')\geq 7$ for all distinct $Q, Q'\in \cQ$. Then
$$|\cQ| \geq (\lfloor 20\eps_4n/7 \rfloor - 2\eps_4n)/2 > 4\eps_3n>d+s+t.$$
If necessary, reverse the order of all vertices in $C$ so that the paths in $\cQ$ are forward paths.
Apply Proposition~\ref{prop:balance} to find a matching $M\subseteq E(B\cup T,B)$ of size $d$ and write $M=\{e_1, \dots, e_m, f_{m+1}, \dots, f_d\}$, where $e_i\in E(B)$ and $f_i\in E(T,B)$. Map the initial vertex of $P$ to any vertex in $A$. We will greedily find a copy of $P$ in $G_M$ which covers $M$ and $S\cup T$ as follows. 

Note that, by (Q\ref{Q7}), we can cover each edge $f_i\in M$ by a forward path of the form $BTB$. By (Q\ref{Q6}), each of the vertices in $S$ can be covered by a forward path of the form $ASB$. Similarly, (Q\ref{Q7}) allows us to find a forward path of the form $BTA$ covering each vertex in $T$. Moreover, note that (Q\ref{Q1})--(Q\ref{Q4}) allow us to find a path of length three of any orientation between any pair of vertices $x\in A$ and $y\in B$ using only edges from $E(A,B) \cup E(B,A)$. So we can find a copy of $P$ which covers every edge in $M$ and every vertex in $(S\cup T)\setminus V(M)$ by a copy of a path in $\cQ$ and which has the form
$$(A*BB)^m(A*BTB)^{d-m}(A*ASB)^s(A*BT)^{t-d+m}A*X,$$
where $X\in \{A,B\}$. We may assume that $X=A$ by extending the path $P$ by one vertex if necessary. Let $P^G$ denote this copy of $P$ in $G$.

Now (EC\ref{EC1}) and (EC\ref{EC2}) hold. It remains to check (EC\ref{EC3}). 
Observe that $P^G$ contains no repeated $A$s and exactly $d$ repeated $B$s, these occur in the subpath of $P^G$ of the form $(A*BB)^m(A*BTB)^{d-m}$. By (\ref{eqn:repeats}), we see that $$|B\setminus V(P^G)|-|A\setminus V(P^G)|= 1,$$ so (EC\ref{EC3}) is satisfied. Hence $P^G$ forms an exceptional cover.
\end{proof}

\subsection{Finding an exceptional cover when $C$ has many sink vertices, $\sigma(C) \geq\eps_4n$}

When $C$ is far from being consistently oriented, we use sink and source vertices to cover the vertices in $S\cup T$. A natural approach would be to try to cover the vertices in $S\cup T$ by paths of the form $ASA$ and $BTB$ whose central vertex is a sink or by paths of the form $ATA$ and $BSB$ whose central vertex is a source. In essence, this is what we will do, but there are some technical issues we will need to address. The most obvious is that each time we cover a vertex in $S$ or $T$ by a path of one of the above forms, we will introduce a repeated $A$ or a repeated $B$, so we will need to cover the exceptional vertices in a ``balanced'' way.

Let $P$ be a subpath of $C$ and let $m$ be the number of sink vertices in $P$. Suppose that $P_1,P_2,P_3$ is a partition of $P$ into internally disjoint paths such that $P=(P_1P_2P_3)$. We say that $P_1, P_2, P_3$ is a \emph{useful tripartition of $P$} if there exist $\cQ_i\subseteq V(P_i)$ such that:
\begin{itemize}
\item $P_1$ and $P_2$ have even length;
\item $|\cQ_i|\geq \lfloor m/12 \rfloor$ for $i=1,2,3$;
\item all vertices in $\cQ_1 \cup \cQ_3$ are sink vertices and are an even distance apart;
\item all vertices in $\cQ_2$ are source vertices and are an even distance apart.%
\end{itemize}
Note that a useful tripartition always exists. We say that $\cQ_1, \cQ_2, \cQ_3$ are \emph{sink/source/sink sets} for the tripartition $P_1, P_2, P_3$. We say that a subpath $L\subseteq P_2$ is a \emph{link} if $L$ has even length and, if, writing $x$ for the initial vertex and $y$ for the the final vertex of $L$, the paths $(P_2x)$ and $(yP_2)$ each contain at least $|\cQ_2|/3$ elements of $\cQ_2$.

\begin{prop}\label{prop:sinksource}
Let $1/n \ll \eps \ll \eta \ll \tau \leq 1$. Let $G$ be a digraph on $n$ vertices and let $A,B,S,T$ be a partition of $V(G)$. Let $S_A, S_B$ be disjoint subsets of $S$ and $T_A, T_B$ be disjoint subsets of $T$. Let $a:=|A|$, $b:=|B|$, $s_A:= |S_A|$, $s_B:=|S_B|$, $t_A:=|T_A|$, $t_B:=|T_B|$ and let $a_1\in A$. Suppose that:
\begin{enumerate}[\rm(i)]
	\item $a,b\geq \tau n$;
	\item $s_A,s_B,t_A,t_B \leq \eps n$;
	\item $\delta^0(G[A,B]) \geq \eta n$;
	\item $d^\pm _B(x) \geq b -\eps n$ for all but at most $\eps n$ vertices $x \in A$;
	\item $d^\pm _A(x) \geq a -\eps n$ for all but at most $\eps n$ vertices $x \in B$;
	\item $d^-_A(x)\geq \eta n$ for all $x\in S_A$, $d^+_B(x)\geq \eta n$ for all $x\in S_B$, $d^+_A(x)\geq \eta n$ for all $x\in T_A$ and $d^-_B(x)\geq \eta n$ for all $x\in T_B$.
\end{enumerate}
Suppose that $P$ is a path of length at most $\eta^2 n$ which contains at least $200\eps n$ sink vertices. Let $P_1, P_2, P_3$ be a useful tripartition of $P$ with sink/source/sink sets $\cQ_1, \cQ_2, \cQ_3$. Let $L\subseteq P_2$ be a link. Suppose that $G\setminus (S_A\cup S_B\cup T_A\cup T_B)$ contains a copy $L^G$ of $L$ which is an $AB$-path if $d_C(P,L)$ is even and a $BA$-path otherwise. Let $r_A$ be the number of repeated $A$s in $L^G$ and $r_B$ be the number of repeated $B$s in $L^G$. Let $G'$ be the graph with vertex set $V(G)$ and edges $$E(A,B\cup S_A)\cup E(B,A\cup T_B) \cup E(T_A, A) \cup E(S_B,B) \cup E(L^G).$$ Then $G'$ contains a copy $P^G$ of $P$ such that:
\begin{itemize}
	\item $L^G\subseteq P^G$;
	\item $P^G$ covers $S_A, S_B, T_A, T_B$;
	\item $a_1$ is the initial vertex of $P^G$;
	\item The final vertex of $P^G$ lies in $B$ if $P$ has even length and $A$ if $P$ has odd length;
	\item $P^G$ has $s_A+t_A+r_A$ repeated $A$s and $s_B+t_B+r_B$ repeated $B$s.
\end{itemize}
\end{prop}

\begin{proof}
We may assume, without loss of generality, that the initial vertex of $P$ lies in $\cQ_1$. If not, let $x$ be the first vertex on $P$ lying in $\cQ_1$ and greedily embed the initial segment $(Px)$ of $P$ starting at $a_1$ using edges in $E(A,B)\cup E(B,A)$. Let $a_1'$ be the image of $x$. We can then use symmetry to relabel the sets $A,B, S_A, S_B, T_A, T_B$, if necessary, to assume that $a_1'\in A$.

We will use (vi) to find a copy of $P$ which covers the vertices in $S_A\cup T_B$ by sink vertices in $\cQ_1\cup \cQ_3$ and the vertices in $S_B\cup T_A$ by source vertices in $\cQ_2$. We will use that $|\cQ_i|\geq 15\eps n$ for all $i$ and also that (iii)--(v) together imply that $G'$ contains a path of length three of any orientation between any pair of vertices in $x\in A$ and $y\in B$. Consider any $q_1\in \cQ_1$ and $q_2\in \cQ_2$. The order in which we cover the vertices will depend on whether $d_C(q_1,q_2)$ is even or odd (note that the parity of $d_C(q_1,q_2)$ does not depend on the choice of $q_1$ and $q_2$).

Suppose first that $d_C(q_1,q_2)$ is even. We find a copy of $P$ in $G'$ as follows. Map the initial vertex of $P$ to $a_1$. Then greedily cover all vertices in $T_B$ so that they are the images of sink vertices in $\cQ_1$ using a path $P_1^G$ which is isomorphic to $P_1$ and has the form $(A*BT_BB)^{t_B}A*A$.
Let $x_L$ be the initial vertex of $L$ and $y_L$ be the final vertex. Let $x_L^G$ and $y_L^G$ be the images of $x_L$ and $y_L$ in $L^G$. Cover all vertices in $S_B$ so that they are the images of source vertices in $\cQ_2$ using a path isomorphic to $(P_2x_L)$ which starts from the final vertex of $P_1^G$ and ends at $x_L^G$. This path has the form $(A*BS_BB)^{s_B}A*X$, where $X:=A$ if $d_C(P,L)$ is even and $X:=B$ if $d_C(P,L)$ is odd. Now use the path $L^G$. Next cover all vertices in $T_A$ so that they are the images of source vertices in $\cQ_2$ using a path isomorphic to $(y_LP_2)$ whose initial vertex is $y_L^G$. This path has the form $Y*A(B*AT_AA)^{t_A}B*B$, where $Y:=B$ if $d_C(P,L)$ is even and $Y:=A$ if $d_C(P,L)$ is odd.%
\COMMENT{We use that the parity of the lengths of $(P_2x_L)$ and $(y_LP_2)$ is the same. Since $P_1$ and $P_2$ have even length we note that the parity of the lengths of $P_3$ and $P$ is the same. Finally, the vertices in $\cQ_3$ have even distance to the initial vertex of $P_3$.}
 Let $P_2^G$ denote the copy of $P_2$ obtained in this way. Finally, starting from the final vertex of $P_2^G$, find a copy of $P_3$ which covers all vertices in $S_A$ by sink vertices in $\cQ_3$ and has the form $(B*AS_AA)^{s_A}B*B$ if $P$ (and thus also $P_3$) has even length and $(B*AS_AA)^{s_A}B*A$ if $P$ (and thus also $P_3$) has odd length. If $d_C(q_1,q_2)$ is odd, we find a copy of $P$ which covers $T_B$, $T_A$, $V(L^G)$, $S_B$, $S_A$ (in this order) in the same way.%
\COMMENT{Map the initial vertex of $P$ to $a_1$. Then greedily cover all vertices in $T_B$ so that they are the images of sink vertices in $\cQ_1$ using a path $P_1^G$ which is isomorphic to $P_1$ and has the form $(A*BT_BB)^{t_B}A*A$.
Let $x_L$ be the initial vertex of $L$ and $y_L$ be the final vertex. Let $x_L^G$ and $y_L^G$ be the images of $x_L$ and $y_L$ in $L^G$. Cover all vertices in $T_A$ so that they are the images of source vertices in $\cQ_2$ using a path isomorphic to $(P_2x_L)$ which starts from the final vertex of $P_1^G$ and ends at $x_L^G$. This path has the form $A(B*AT_AA)^{t_A}B*X$, where $X:=A$ if $d_C(P,L)$ is even and $X:=B$ if $d_C(P,L)$ is odd. Now use the path $L^G$. Next cover all vertices in $S_B$ so that they are the images of source vertices in $\cQ_2$ using a path isomorphic to $(y_LP_2)$ whose initial vertex is $y_L^G$. This path has the form $Y*B(A*BS_BB)^{s_B}A*B$, where $Y:=B$ if $d_C(P,L)$ is even and $Y:=A$ if $d_C(P,L)$ is odd. Let $P_2^G$ denote the copy of $P_2$ obtained in this way. Finally, starting from the final vertex of $P_2^G$, find a copy of $P_3$ which covers all vertices in $S_A$ by sink vertices in $\cQ_3$ and has the form $(B*AS_AA)^{s_A}B*B$ if $P$ has even length and $(B*AS_AA)^{s_A}B*A$ if $P$ has odd length.}
Observe that $P^G$ has $s_A+t_A+r_A$ repeated $A$s and $s_B+t_B+r_B$ repeated $B$s, as required.
\end{proof}

We are now in a position to find an exceptional cover. The proof splits into a number of cases and we will require the assumption that $C$ is not antidirected. We will need a matching found using Proposition~\ref{prop:balance} and a careful assignment of the remaining vertices in $S\cup T$ to sets $S_A, S_B, T_A$ and $T_B$ to ensure that the path found by Proposition~\ref{prop:sinksource} leaves a balanced number of vertices in $A$ and $B$ uncovered.

\begin{lemma}\label{lem:excover2}
Suppose $1/n \ll \eps_3 \ll \eps_4 \ll 1$.
Let $G$ be a digraph on $n$ vertices with $\delta^0(G) \geq n/2$. Suppose $A, B, S, T$ is a partition of $V(G)$ satisfying (Q\ref{Q*})--(Q\ref{Q8}). If $C$ is an oriented cycle on $n$ vertices, $C$ is not antidirected and $\sigma(C) \geq\eps_4n$, then there is an exceptional cover $P$ of $G$ of length at most $2\eps_4 n$.
\end{lemma}

\begin{proof}
Let $d:=b-a$, $k:=t-s$ and $r:=s+t$.
Since $\sigma(C) \geq \eps_4n$, we can use an averaging argument to guarantee a subpath $Q'$ of $C$ of length at most $\eps_4n$ such that $Q'$ contains at least $2\sqrt{\eps_3}n$ sink vertices.%
\COMMENT{To see this, we partition $C$ into subpaths of length $\eps_4 n$. The expected number of sink vertices in each subpath is at least
$(\eps_4n-2)\eps_4n/n\geq 2\sqrt{\eps_3}n$.}
Let $Q$ be an initial subpath of $Q'$ which has odd length and contains $\sqrt{\eps_3}n$ sink vertices.

\medskip

\noindent \textbf{Case 1: }\emph{$a<b$ or $s<t$.}

We will find disjoint sets of vertices $S_A,S_B, T_A, T_B$, of sizes $s_A, s_B, t_A, t_B$ respectively, and a matching $M'=E\cup E'$ (where $E$ and $E'$ are disjoint) such that the following hold:
\begin{enumerate}[(E1)]
\item $S_A\cup S_B=S$ and $T_A\cup T_B=T\setminus V(E')$;\label{E1}
\item $E\subseteq E(B)$, $|E|\leq d$;\label{E2a}
\item $E' \subseteq E(B\cup T, B) \cup E(A,A\cup T)$ and $1\leq |E'|\leq 2$;\label{E2b}
\item If $p:=|E'\cap E(B)|-|E'\cap E(A)|$, then $s_A+t_A+d=s_B+t_B+p+|E|$.\label{E3}
\end{enumerate}

We find sets satisfying (E\ref{E1})--(E\ref{E3}) as follows.
Suppose first that $n$ is odd. Note that we can find a matching $M\subseteq E(B\cup T,B)$ of size $d+1$. Indeed, if $a<b$ then $M$ exists by Proposition~\ref{prop:balance} and if $a=b$, and so $s<t$, we use that $a+s<n/2$ and $\delta^0(G) \geq n/2$ to find $M$ of size $d+1=1$.
 Fix one edge $e\in M$ and let $E':=\{e\}$. There are $r':=r-|V(E')\cap T|$ vertices in $S\cup T$ which are not covered by $E'$. Set $d':=\min\{r',d-p\}$ and let $E\subseteq(M\setminus E') \cap E(B)$ have size $d-p-d'$.%
\COMMENT{$M\setminus E'$ can contain at most $r'$ $TB$-edges, so must have at least $d-r'$ $B$-edges. Note that $d-p-d'=\max \{d-p-r', 0\}\leq d-r'$ as $p\geq 0$.} 

Suppose that $n$ is even.  If $a<b$, by Proposition~\ref{prop:balance}, we find a matching $M$ of size $d+2$ in $E(B\cup T, B)$. Fix two edges $e_1, e_2 \in M$ and let $E':=\{e_1, e_2\}$. Choose $r'$, $d'$ and $E$ as above.

If $n$ is even and $a=b$, then $a+s = b+s = (n-k)/2 \leq n/2-1$. So $d^+_{A\cup T}(x) \geq k/2$ for each $x\in A$ and $d^-_{B \cup T}(x)\geq k/2$ for each $x\in B$. Either we can find a matching $M$ of size two in $E(B\cup T, B) \cup E(A, A\cup T)$ or $t=s+2$ and there is a vertex $v \in T$ such that $A \subseteq N^-(v)$ and $B \subseteq N^+(v)$. In the latter case, move $v$ to $S$ to get a new partition satisfying (Q\ref{Q*})--(Q\ref{Q8}) and the conditions of Case~2.  So we will assume that the former holds. Let $E':=M$, $E:=\emptyset$, $r':=r-|V(E')\cap T|$ and $d':=-p$.

In each of the above cases, note that $d'\equiv r' \mod 2$%
\COMMENT{since $r'=r+|E'\cap E(A)|+|E'\cap E(B)|-|E'|$ and $d\equiv r-|E'| \mod 2$.}
 and $|d'|\leq r'$. So we can choose disjoint subsets $S_A, S_B, T_A, T_B$ satisfying (E\ref{E1}) such that $s_A+t_A=(r'-d')/2$ and $s_B+t_B= (r'+d')/2$. Then (E\ref{E3}) is also satisfied.%
\COMMENT{$s_B+t_B+p+|E|=(r'+d')/2+p+(d-p-d') = (r'-d')/2+d=s_A+t_A+d$.}

We construct an exceptional cover as follows.
Let $L_1$ denote the oriented path of length two whose second vertex is a sink and let $L_2$ denote the oriented path of length two whose second vertex is a source. For each $e\in E'$, we find a copy $L(e)$ of $L_1$ or $L_2$ covering $e$. If $e\in E(A)$ let $L(e)$ be a copy of $L_1$ of the form $AAB$, if $e\in E(B)$ let $L(e)$ be a copy of $L_1$ of the form $ABB$, if $e\in E(A,T)$ let $L(e)$ be a copy of $L_1$ of the form $ATB$ and if $e\in E(T,B)$ let $L(e)$ be a copy of $L_2$ of the form $ATB$. Note that for each $e\in E'$, the orientation of $L(e)$ is the same regardless of whether it is traversed from its initial vertex to final vertex or vice versa. This means that we can embed it either as an $AB$-path or a $BA$-path.

Let $a_1$ be any vertex in $A$ and let $e_1\in E'$. Let $r_A$ and $r_B$ be the number of repeated $A$s and $B$s, respectively, in $L(e_1)$. So $r_A=1$ if and only if $e_1\in E(A)$, otherwise $r_A=0$. Also, $r_B=1$ if and only if $e_1\in E(B)$, otherwise $r_B=0$. 
Consider a useful tripartition $P_1, P_2, P_3$ of $Q$. Let $L\subseteq P_2$ be a link which is isomorphic to $L(e_1)$. Let $x$ denote the final vertex of $Q$. Using Proposition~\ref{prop:sinksource} (with $2\eps_3, \eps_4, 1/4$ playing the roles of $\eps, \eta, \tau$), we find a copy $Q^G$ of $Q$ covering $S_A, S_B, T_A, T_B$ whose initial vertex is $a_1$. Moreover, $L(e_1)\subseteq Q^G \subseteq G_{\{e_1\}} \subseteq G_M$, the final vertex $x^G$ of $Q^G$ lies in $A$, $Q^G$ has $s_A+t_A+r_A$ repeated $A$s and $s_B+t_B+r_B$ repeated $B$s. If $|E'|=2$, let $e_2\in E'\setminus\{e_1\}$. Let $Q'':=(xQ')$. Let $y$ be the second source vertex in $Q''$ if $e_2\in E(T,B)$ and the second sink vertex in $Q''$ otherwise. Let $y^-$ be the vertex preceding $y$ on $C$, let $y^+$ be the vertex following $y$ on $C$ and let $q:=d_C(x,y^-)$.  Find a path in $G$ whose initial vertex is $x^G$ which is isomorphic to $(Q''y^-)$ and is of the form $A*A$  if $q$ is even and $A*B$ if $q$ is odd, such that the final vertex of this path is an endvertex of $L(e_2)$. Then use the path $L(e_2)$ itself. Let $Z:=B$ if $q$ is even and $Z:=A$ if $q$ is odd.  Finally, extend the path to cover all edges in $E$ using a path of the form $Z*B(A*ABB)^{|E|}A$ which is isomorphic to an initial segment of $(y^+Q'')$. Let $P$ denote the resulting extended subpath of $C$, so $Q\subseteq P \subseteq Q'$. Let $P^G$ be the copy of $P$ in $G_M$.

Note that (EC\ref{EC1}) and (EC\ref{EC2}) hold. Each repeated $A$ in $P^G$ is either a repeated $A$ in $Q^G$ or it occurs when $P^G$ uses $L(e_2)$ in the case when $e_2\in E(A)$. Similarly, each repeated $B$ in $P^G$ is either a repeated $B$ in $Q^G$ or it occurs when $P^G$ uses $L(e_2)$ in the case when $e_2\in E(B)$ or when $P^G$ uses an edge in $E$. Substituting into (\ref{eqn:repeats}) and recalling (E\ref{E3}) gives
\begin{align*}
|B \setminus V(P^G)|-|A\setminus V(P^G)| =& b-a-(s_B+t_B+|E|+|E'\cap E(B)|)+ (s_A+t_A+|E'\cap E(A)|)+1\\
=&d -(s_B+t_B+|E|)-p+(s_A+t_A)+1=1.
\end{align*}
So (EC\ref{EC3}) is satisfied and $P^G$ is an exceptional cover.

\medskip

\noindent \textbf{Case 2: } \emph{$a=b$ and $s=t$.}

If $s=t=0$ then any path consisting of one vertex in $A$ is an exceptional cover. So we will assume that $s,t\geq 1$. 
We say that $C$ is \emph{close to antidirected} if it contains an antidirected subpath of length $500\eps_3n$.

\medskip

\noindent \textbf{Case 2.1: }\emph{$C$ is close to antidirected.}

If there is an edge $e\in E(T,B)\cup E(B, S) \cup E(S, A) \cup E(A,T)$ then we are able to find an exceptional cover in the graph $G_{\{e\}}$.  We illustrate how to do this when $e=t_1b_1\in E(T,B)$, the other cases are similar.%
\COMMENT{If $e=a_1t_1\in E(A,T)$ then the proof is almost identical, only difference being that the link is a copy of $L_1$ (oriented path of length two whose second vertex is a sink).\\
Suppose $e=s_1a_1\in E(S,A)$. Let $t_1\in T$. If initial edge of $P$ is forward, let $P'$ consist of first two edges of $P$ and let $(P')^G$ be a forward path of the form $B\{t_1\}A$. If initial edge is backward, let $P'$ be the first three edges of $P$ and find a copy $(P')^G$ of $P'$ of the form $A\{t_1\}BA$. Let $P'':=P\setminus P''$. Find a copy $L^G$ of $L_2$ which is of the form $ASB$ and uses $e$. Let $L\subseteq P_2$ be a link isomorphic to $L_1$. Set $S_A:=S\setminus\{s_1\}$, $T_B:=T \setminus\{t_1\}$ and $S_B, T_A:=\emptyset$. Apply Proposition~\ref{prop:sinksource}, as above. If $e=b_1s_1\in E(B,S)$, then the proof is nearly identical, uses copy of $L_1$ as the link instead.}
Since $C$ is close to but not antidirected, it follows that $C$ contains a path $P$ of length $500\eps_3n$ which is antidirected except for the initial two edges which are oriented consistently. Let $s_1 \in S$.  If the initial edge of $P$ is a forward edge, let $P'$ be the subpath of $P$ consisting of the first three edges of $P$ and find a copy $(P')^G$ of $P'$ in $G$ of the form $A\{s_1\}BA$. If the initial edge of $P$ is a backward edge, let $P'$ consist of the first two edges of $P$ and let $(P')^G$ be a backward path of the form $B\{s_1\}A$. Let $P''$ be the subpath of $P$ formed by removing from $P$ all edges in $P'$. Let $x^G\in A$ be the final vertex of $(P')^G$. Set $S_A:=S\setminus\{s_1\}$, $T_B:=T \setminus\{t_1\}$ and $S_B, T_A:=\emptyset$. Let $P_1,P_2, P_3$ be a useful tripartition of $P''$. As in Case~1, let $L_2$ denote the oriented path of length two whose second vertex is a source. Let $L\subseteq P_2$ be a link which is isomorphic to $L_2$ and map $L$ to a path $L^G$ of the form $BTA$ which uses the edge $t_1b_1$.  We use Proposition~\ref{prop:sinksource} to find a copy $(P'')^G$ of $P''$ which uses $L^G$, covers $S_A\cup T_B$ and whose initial vertex is mapped to $x^G$. Moreover, the final vertex of $P''$ is mapped to $A\cup B$ and $(P'')^G$ has $s_A=s-1$ repeated $A$s and $t_B=t-1$ repeated $B$s. Let $P^G$ be the path $(P')^G\cup (P'')^G$. Then $P^G$ satisfies (EC\ref{EC1}) and we may assume that (EC\ref{EC2}) holds, by adding a vertex in $A$ as a new initial vertex and/or final vertex if necessary. The repeated $A$s and $B$s in $P^G$ are precisely the repeated $A$s and $B$s in $(P'')^G$. Therefore, \eqref{eqn:repeats} implies that (EC\ref{EC3}) holds and $P^G$ forms an exceptional cover.

Let us suppose then that $E(T,B)\cup E(B, S)\cup E(S, A) \cup E(A,T)$ is empty. If $S=\{s_1\}, T=\{t_1\}$ then, since $\delta^0(G) \geq n/2$, $G$ must contain the edge $s_1t_1$ and edges $a_1s_1,b_1t_1$ for some $a_1\in A, b_1\in B$.  Since $C$ is not antidirected but has many sink vertices we may assume that $C$ contains a subpath $P=(uvxyz)$ where $uv, vx, yx\in E(C)$. We use the edges $a_1s_1, s_1t_1, b_1t_1$, as well as an additional $AB$- or $BA$-edge, to find a copy $P^G$ of $P$ in $G$ of the form $ASTBA$. The path $P^G$ forms an exceptional cover.

If $s=t=2$ and $e(S)=e(T)=2$, we find an exceptional cover as follows. Write $S=\{s_1, s_2\}$, $T=\{t_1,t_2\}$. We have that $s_is_j, t_it_j \in E(G)$ for all $i\neq j$. Note that $C$ is not antidirected, so $C$ must contain a path of length six which is antidirected except for its initial two edges which are consistently oriented.  Suppose first that the initial two edges of $P$ are forward edges. Let $a_1\in A$ be an inneighbour of $s_1$. Note that $s_2$ has an inneighbour in $T$, without loss of generality $t_1$. Let $b_1\in B$ be an inneighbour of $t_2$ and $a_2\in A$ be an outneighbour of $b_1$. We find a copy $P^G$ of $P$ which has the form $ASSTTBA$ and uses the edges $a_1s_1, s_1s_2, t_1s_2, t_1t_2, b_1t_2, b_1a_2$, in this order. If the initial two edges of $P$ are backward, we instead find a path of the form $ATTSSBA$. Note that in both cases, $P^G$ satisfies (EC\ref{EC1}) and (EC\ref{EC2}). $P^G$ has no repeated $A$s and $B$s and \eqref{eqn:repeats} implies that (EC\ref{EC3}) holds. So $P^G$ forms an exceptional cover.

So let us assume that $s,t\geq 2$ and, additionally, $e(S)+e(T)<4$ if $s=2$. There must exist two disjoint edges $e_1=t_1s_1$, $e_2=s_2t_2$ where $s_1,s_2 \in S$ and $t_1, t_2 \in T$ (since $\delta^0(G)\geq n/2$ and $E(T,B)\cup E(B, S)\cup E(S, A) \cup E(A,T)=\emptyset$).%
\COMMENT{If $s=t=2$, without loss of generality, $e(S)<2$. So we can write $S=\{s_1, s_2\}$ such that $s_1s_2\not\in E(G)$. Note that $s_1$ must have at least two outneighbours in $T$. Since $s_2$ must have at least one inneighbour in $T$, we can find the desired edges. If $s,t\geq 3$, we note that each vertex in $S$ has at least one outneighbour in $T$ and each vertex in $T$ has at least one inneighbour in $S$. By K\"onig's theorem, we find two disjoint $ST$-edges $e_1'$ and $e_2'$. Choose any $s_1\in S$ which does not lie on $e_1'$ or $e_2'$. Note that $s_1$ has an inneighbour in $T$ and this vertex can lie on at most one of $e_1',e_2'$. Hence we find the desired edges.}
We use these edges to find an exceptional cover as follows. We let $S_A:= S\setminus\{s_1,s_2\}$, $T_B:=T\setminus\{t_1, t_2\}$, $s_A:=|S_A|$ and $t_B:=|T_B|$. We use $e_1$ and $e_2$ to find an antidirected path $P^G$ which starts with a backward edge and is of the form
$$A\{t_1\}\{s_1\}A(B*AS_AA)^{s_A}B*B\{s_2\}\{t_2\}B(A*BT_BB)^{s_B}A.$$ The length of $P^G$ is less than  $500\eps_3n$. So, as $C$ is close to antidirected, $C$ must contain a subpath isomorphic to $P^G$. We claim that $P^G$ is an exceptional cover. Clearly, $P^G$ satisfies (EC\ref{EC1}) and (EC\ref{EC2}). For (EC\ref{EC3}), note that $P^G$ contains an equal number of repeated $A$s and repeated $B$s. Then (\ref{eqn:repeats}) implies that $|B\cap V(P^G)|=|A\cap V(P^G)|+1$.

\medskip

\noindent \textbf{Case 2.2: }\emph{$C$ is far from antidirected.}

Recall that $Q$ is a subpath of $C$ of length at most $\eps_4n$ containing at least $\sqrt{\eps_3}n$ sink vertices. Let $\cQ$ be a maximum collection of sink vertices in $Q$ such that all vertices in $\cQ$ are an even distance apart, then $|\cQ|\geq \sqrt{\eps_3}n/2$. Partition the path $Q$ into $11$ internally disjoint subpaths so that $Q=(P_1P_1'P_2P_2'\dots P_5P_5'P_6)$ and each subpath contains at least $300\eps_3n$ elements of $\cQ$. Note that each $P_i'$ has length greater than $500\eps_3n$ and so is not antidirected, that is, each $P_i'$ must contain a consistently oriented subpath $P_i''$ of length two. At least three of the $P_i''$ must form a consistent set. Thus there must exist $i<j$ such that $d_C(P_i'', P_j'')$ is even and $\{P_i'',P_j''\}$ is consistent. We may assume, without loss of generality, that $P_i'', P_j''$ are forward paths and that the second vertex of $P_i$ is in $\cQ$. Let $P$ be the subpath of $Q$ whose initial vertex is the initial vertex of $P_i$ and whose final vertex is the final vertex of $P_j''$.

We will find an exceptional cover isomorphic to $P$ as follows. Choose $s_1\in S$ and $t_1 \in T$ arbitrarily. Set $S_A:=S\setminus \{s_1\}$ and $T_B:=T\setminus \{t_1\}$. Map the initial vertex of $P$ to $A$. We find a copy of $P$ which maps each vertex in $S_A$ to a sink vertex in $P_i$ and each vertex in $T_B$ to a sink vertex in $P_j$. If $d_C(P_i, P_i'')$ is even, $P_i''$ is mapped to a path $L'$ of the form $A\{s_1\}B$ and $P_j''$ is mapped to a path $L''$ of the form $B\{t_1\}A$. If $d_C(P_i, P_i'')$ is odd, $P_i''$ is mapped to a path $L'$ of the form $B\{t_1\}A$ and $P_j''$ is mapped to a path $L''$ of the form $A\{s_1\}B$. Thus, if $d_C(P_i, P_i'')$ is even, we obtain a copy $P^G$ which starts with a path of the form $A(B*AS_AA)^{s_A}B*A$, then uses $L'$ and continues with a path of the form $B*B(A*BT_BB)^{t_B}A*B$. Finally, the path uses $L''$. The case when $d_C(P_i, P_i'')$ is odd is similar.%
\COMMENT{$d_C(P_i, P_i'')$ is odd: we obtain a copy $P^G$ which starts with a path of the form $A(B*AS_AA)^{s_A}B*B$, then uses $L'$ and continues with a path of the form $A*B(A*BT_BB)^{t_B}A*A$. Finally, the path uses $L''$.}
 (EC\ref{EC1}) holds and we may assume that (EC\ref{EC2}) holds by adding one vertex to $P$ if necessary. Note that $P^G$ contains an equal number of repeated $A$s and $B$s, so (\ref{eqn:repeats}) implies that (EC\ref{EC3}) holds and $P^G$ is an exceptional cover.
\end{proof}

\subsection{Finding a copy of $C$}

Proposition~\ref{prop:excover1} and Lemma~\ref{lem:excover2} allow us to find a short exceptional cover for any cycle which is not antidirected. We complete the proof of Lemma~\ref{lem:AB} by extending this path to cover the small number of vertices of low degree remaining in $A$ and $B$ and then applying Proposition~\ref{prop:completepath}.

\begin{proofof}\textbf{Lemma~\ref{lem:AB}.}
Let $P$ be an exceptional cover of $G$ of length at most $21\eps_4n$, guaranteed by Proposition~\ref{prop:excover1} or Lemma~\ref{lem:excover2}. Let 
\begin{align*}
X:= \{v\in A : d^+_B(v)< n/2-\eps_3 n \text{ or } d^-_B(v) < n/2-\eps_3 n\} &\text{ and}\\
Y:= \{v\in B : d^+_A(v)< n/2-\eps_3 n \text{ or } d^-_A(v) < n/2-\eps_3 n\}&.
\end{align*}
 (Q\ref{Q3}) and (Q\ref{Q4}) together imply that $|X\cup Y|\leq 2\eps_3 n$. Together with (Q\ref{Q2}), this allows us to cover the vertices in $X\cup Y$ by any orientation of a path of length at most $\eps_4 n$. So we can extend $P$ to cover the remaining vertices in $X\cup Y$ (by a path which alternates between $A$ and $B$). Let $P'$ denote this extended path. Thus $|P'| \leq 22\eps_4n$. Let $x$ and $y$ be the endvertices of $P'$. We may assume that $x, y \in A\setminus X$. Let $A':=(A\setminus V(P'))\cup \{x,y\}$ and $B':=B\setminus V(P')$ and consider $G':=G[A',B']$. Note that $|A'|=|B'|+1$ by (EC\ref{EC3}) and 
$$\delta^0(G') \geq n/2-\eps_3 n-22\eps_4n \geq (7|B'|+2)/8.$$ 
Thus, by Proposition~\ref{prop:completepath}(ii), $G'$ has a Hamilton path of any orientation between $x$ and $y$ in $G$. We combine this path with $P'$, to obtain a copy of $C$.
\end{proofof}

\section{$G$ is $ABST$-extremal}\label{sec:ABST}

In this section we prove that Theorem~\ref{thm:main} holds for all $ABST$-extremal graphs. When $G$ is $ABST$-extremal, the sets $A$, $B$, $S$ and $T$ are all of significant size; $G[S]$ and $G[T]$ look like cliques and $G[A,B]$ resembles a complete bipartite graph. The proof will combine ideas from Sections~\ref{sec:ST}~and~\ref{sec:AB}.

\begin{lemma}\label{lem:ABST}
Suppose that $1/n \ll \eps \ll \eps_1 \ll \eta_1 \ll \tau \ll 1.$
Let $G$ be a digraph on $n$ vertices with $\delta^0(G)\geq n/2$ and assume that $G$ is $ABST$-extremal. If $C$ is any orientation of a cycle on $n$ vertices which is not antidirected, then $G$ contains a copy of $C$.
\end{lemma}

We will again split the proof into two cases, depending on how many changes of direction $C$ contains.  In both cases, the first step is to find an exceptional cover (defined in Section~\ref{sec:AB}) which uses only a small number of vertices from $A\cup B$.

\subsection{Finding an exceptional cover when $C$ has few sink vertices, $\sigma(C) <\eps_2n$}\label{sec:ABST1}

The following lemma allows us to find an exceptional cover when $C$ is close to being consistently oriented. The two main components of the exceptional cover are a path $P_S\subseteq G[S]$ covering most of the vertices in $S$ and another path $P_T\subseteq G[T]$ covering most of the vertices in $T$. We are able to find $P_S$ and $P_T$ because $G[S]$ and $G[T]$ are almost complete. A shorter path follows which uses long runs (recall that a long run is a consistently oriented path of length $20$) and a small number of vertices from $A\cup B$ to cover any remaining vertices in $S\cup T$. We use edges found by Proposition~\ref{prop:dedges} to control the number of repeated~$A$s and $B$s on this path.

\begin{lemma}\label{lem:ABST1}
Suppose $1/n \ll \eps \ll \eps_1 \ll \eps_2\ll \eta_1 \ll \tau \ll 1$.
Let $G$ be a digraph on $n$ vertices with $\delta^0(G) \geq n/2$. Suppose $A, B, S, T$ is a partition of $V(G)$ satisfying (R\ref{R*})--(R\ref{R9}). Let $C$ be an oriented cycle on $n$ vertices. If $\sigma(C) <\eps_2n$, then $G$ has an exceptional cover $P$ such that $|V(P)\cap(A\cup B)|\leq 2\eta_1^2 n$.
\end{lemma}

\begin{proof}
Let $s^*:=s-\lceil\eps_2 n\rceil$ and $d:=b-a$. Define $S'\subseteq S$ to consist of all vertices $x\in S$ with $d^+_{B\cup S}(x) \geq b+s-\eps^{1/3}n$ and $d^-_{A\cup S}(x)\geq a+s-\eps^{1/3}n$. Define $T'\subseteq T$ similarly. Note that $|S\setminus S'|, |T\setminus T'| \leq \eps_1n$ by (R\ref{R8}) and (R\ref{R9}).

We may assume that the vertices of $C$ are labelled so that the number of forward edges is at least the number of backward edges. Let $Q\subseteq C$ be a forward path of length two, this exists since $\sigma(C)<\eps_2 n$. If $C$ is not consistently oriented, we may assume that $Q$ is immediately followed by a backward edge. Define $e_1, e_2, e_3\in E(C)$ such that $d_C(e_1,Q)=s^*$, $d_C(Q, e_2)=s^*+1$, $d_C(Q,e_3)=2$. Let $P_0:=(e_1Ce_2)$. 

If at least one of $e_1, e_2$ is a forward edge, define paths $P_T$ and $P_S$ of order $s^*$ so that $P_0=(e_1P_TQP_Se_2)$. In this case, map $Q$ to a path $Q^G$ in $G$ of the form $T'AS'$. If $e_1$ and $e_2$ are both backward edges, our choice of $Q$ implies that $e_3$ is also a backward edge. Let $P_T$ and $P_S$ be defined so that $P_0=(e_1P_TQe_3P_Se_2)$. So $|P_T|=s^*$ and $|P_S|=s^*-1$. In this case, map $(Qe_3)$ to a path $Q^G$ of the form $T'ABS'$.

Let $p_T:=|P_T|$ and $p_S:=|P_S|$. Our aim is to find a copy $P_0^G$ of $P_0$ which maps $P_S$ to $G[S]$ and $P_T$ to $G[T]$. We will find $P_0^G$ of the form $F$ as given in Table~\ref{table1}.
\begin{table}[h]
	\begin{tabular}{|c|c|c|c|c|}
	\hline	
	$e_1$& forward  & forward & backward & backward\\
	$e_2$& forward & backward & forward & backward\\
	\hline
	\rule{0pt}{11pt} $F$ & $B(T)^{p_T}A(S)^{p_S}B$ & $B(T)^{p_T}A(S)^{p_S}A$ & $A(T)^{p_T}A(S)^{p_S}B$ & $A(T)^{p_T}AB(S)^{p_S}A$\\
	\hline
	\end{tabular}
	\caption{Proof of Lemma~\ref{lem:ABST1}: $P_0^G$ has form $F$.}\label{table1}
\end{table}
Let $M$ be a set of $d+1$ edges in $E(T,B \cup S) \cup E(B, S)$ guaranteed by Proposition~\ref{prop:dedges}. We also define a subset $M'$ of $M$ which we will use to extend $P_0^G$ to an exceptional cover. If $e_1, e_2$ are both forward edges, choose $M' \subseteq M$ of size $d$. Otherwise let $M':=M$. Let $d':=|M'|$. Let $M_1'$ be the set of all edges in $M'$ which are disjoint from all other edges in $M'$ and let $d_1':=|M_1'|$. So $M'\setminus M_1'$ consists of $(d'-d_1')/2=:d_2'$ disjoint consistently oriented paths of the form $TBS$.

We now fix copies $e_1^G$ and $e_2^G$ of $e_1$ and $e_2$. If $e_1$ is a forward edge, let $e_1^G$ be a $BT'$-edge, otherwise let $e_1^G$ be a $T'A$-edge. If $e_2$ is a forward edge, let $e_2^G$ be a $S'B$-edge, otherwise let $e_2^G$ be an $AS'$-edge. Let $t_1$ be the endpoint of $e_1^G$ in $T'$, $s_2$ be the endpoint of $e_2^G$ in $S'$ and let $t_2\in T'$ and $s_1\in S'$ be the endpoints of $Q^G$. Let $v$ be the final vertex of $e_2^G$ and let $X\in \{A,B\}$ be such that $v\in X$.

We now use (R\ref{R4}), (R\ref{R5}), (R\ref{R8}) and (R\ref{R9}) to find a collection $\cP$ of at most $3\eps_1n+1$ disjoint, consistently oriented paths which cover the edges in $M'$ and the vertices in $S\setminus S'$ and $T\setminus T'$. $\cP$ uses each edge $e\in M_1'$ in a forward path $P_e$ of the form $B(S\cup T)^jB$ for some $1\leq j\leq 4$ and $\cP$ uses each path in $M'\setminus M_1'$ in a forward path of the form $BT^jBS^{j'}B$ for some $1\leq j,j'\leq 4$. The remaining vertices in $S\setminus S'$, $T\setminus T'$ are covered by forward paths in $\cP$ of the form $A(S)^jB$ or $B(T)^jA$, for some $1\leq j \leq 3$.

Let $S''\subseteq S \setminus (V(\cP)\cup \{s_1, s_2\})$ and $T''\subseteq T \setminus (V(\cP)\cup \{t_1, t_2\})$ be sets of size at most $2\eps_2n$ so that $|S''|+p_S=|S\setminus V(\cP)|$ and $|T''|+p_T=|T\setminus V(\cP)|$. Note that $S''\subseteq S'$ and $T''\subseteq T'$. So we can cover the vertices in $S''$ by forward paths of the form $ASB$ and we can cover the vertices in $T''$ by forward paths of the form $BTA$. Let $\cP'$ be a collection of disjoint paths thus obtained. Let $P_1$ be the subpath of order $\eta_1^2n$ following $P_0$ on $C$. Note that $P_1$ contains at least $\sqrt{\eps_2}n$ disjoint long runs. Each path in $\cP\cup \cP'$ will be contained in the image of such a long run. (Each forward path in $\cP\cup \cP'$ might be traversed by $P_1^G$ in a forward or backward direction, for example, a forward path of the form $BT^jBS^{j'}B$ could appear in $P_1^G$ as a forward path of the form $BT^jBS^{j'}B$ or a backward path of the form $BS^{j'}BT^jB$.) So we can find a copy $P_1^G$ of $P_1$ starting from $v$ which uses $\cP\cup \cP'$ and has the form 
$$X*AX_1X_2\dots X_{d_1'}Y_1Y_2\dots Y_{d_2'}Z_1Z_2\dots Z_\ell B*Y$$ for some $\ell\geq 0$ and $Y\in \{A,B\}$, where 
\begin{align*}
X_i&\in \{B(S\cup T)^jB*A: 1\leq j\leq 4\},\\
Y_i&\in \{B(S\cup T)^jB(S\cup T)^{j'}B*A: 1\leq j,j'\leq 4\}\hspace{6pt}\text{ and}\\
Z_i&\in\{BA(S\cup T)^jB*A, B(S\cup T)^jA*A:1\leq j\leq 3\}.
\end{align*}

Let $S^*$ be the set of uncovered vertices in $S$ together with the vertices $s_1, s_2$ and let $T^*$ be the set of uncovered vertices in $T$ together with $t_1$ and $t_2$. Write $G_S:=G[S^*]$ and $G_T:=G[T^*]$. Now
$\delta^0(G_T)\geq t-\sqrt{\eps_2}n \geq 7|G_T|/8$ and so $G_T$ has a Hamilton path from $t_1$ to $t_2$ which is isomorphic to $P_T$, by Proposition~\ref{prop:completepath}(i). Similarly, we find a path isomorphic to $P_S$ from $s_1$ to $s_2$ in $G_S$. Altogether, this gives us the desired copy $P_0^G$ of $P_0$ in $G$. Let $P^G:=P_0^GP_1^G$.

We now check that $P^G$ forms an exceptional cover. Clearly (EC1) holds and we may assume that $P^G$ has both endvertices in $A$ (by extending the path if necessary) so that (EC2) is also satisfied. For (EC\ref{EC3}), observe that $P_1^G$ contains exactly $d_1'+2d_2'=d'$ repeated $B$s, these occur in the subpath of the form $X_1X_2\dots X_{d_1'}Y_1Y_2\dots Y_{d_2'}$ covering the edges in $M'$. If $e_1$ and $e_2$ are both forward edges, then, consulting Table~\ref{table1}, we see that $P_0^G$ has no repeated $A$s and that there are no other repeated $A$s or $B$s in $P^G$. Recall that in this case $d'=d$, so \eqref{eqn:repeats} gives $|B\setminus V(P^G)|-|A\setminus V(P^G)| = d-d'+1=1.$ If at least one of $e_1, e_2$ is a backward edge, using Table~\ref{table1}, we see that there is one repeated $A$ in $P_0^G$ and there are no other repeated $A$s or $B$s in $P^G$. In this case, we have $d'=d+1$, so \eqref{eqn:repeats} gives $|B\setminus V(P^G)|-|A\setminus V(P^G)|=d-d'+1+1=1$. Hence $P^G$ satisfies (EC3) and forms an exceptional cover. Furthermore, $|V(P^G)\cap (A\cup B)| \leq 2\eta_1^2 n$.
\end{proof}

\subsection{Finding an exceptional cover when $C$ has many sink vertices, $\sigma(C) \geq \eps_2n$}

In Lemma~\ref{lem:ABST2}, we find an exceptional cover when $C$ contains many sink vertices. The proof will use the following result which allows us to find short $AB$- and $BA$-paths of even length. We will say that an $AB$- or $BA$-path $P$ in $G$ is \emph{useful} if it has no repeated $A$s or $B$s and uses an odd number of vertices from $S\cup T$.

\begin{prop}\label{prop:ABST2links}
Suppose $1/n \ll \eps \ll \eps_1\ll \eta_1 \ll \tau \ll 1.$
Let $G$ be a digraph on $n$ vertices with $\delta^0(G) \geq n/2$. Suppose $A, B, S, T$ is a partition of $V(G)$ satisfying (R\ref{R*})--(R\ref{R9}). Let $L_1$ and $L_2$ be oriented paths of length eight. Then $G$ contains disjoint copies $L_1^G$ and $L_2^G$ of $L_1$ and $L_2$ such that each $L_i^G$ is a useful path. Furthermore, we can specify whether $L_i^G$ is an $AB$-path or a $BA$-path.
\end{prop}

\begin{proof}
Define $S'\subseteq S$ to be the set consisting of all vertices $x\in S$ with $d^\pm_{S}(x)\geq \eta_1 n/2$. Define $T'\subseteq T$ similarly. Note that $|S\setminus S'|, |T\setminus T'|\leq \eps_1n$ by (R\ref{R8}) and (R\ref{R9}). We claim that $G$ contains disjoint edges $e,f\in E(B\cup T,S')\cup E(A\cup S, T')$. Indeed, if $a+s<n/2$ it is easy to find disjoint $e,f\in E(B\cup T,S')$, since $\delta^0(G)\geq n/2$. Otherwise, we must have $a+s=b+t=n/2$ and so each vertex in $S'$ has at least one inneighbour in $B\cup T$ and each vertex in $T'$ has at least one inneighbour in $A\cup S$. Let $G'$ be the bipartite digraph with vertex classes $A\cup S$ and $B\cup T$ and all edges in $E(B\cup T,S')\cup E(A\cup S, T')$. The claim follows from applying K\"onig's theorem to the underlying undirected graph of $G'$.

We demonstrate how to find a copy $L_1^G$ of $L_1$ in $G$ which is an $AB$-path. The argument when $L_1^G$ is a $BA$-path is very similar.%
\COMMENT{For the $BA$-path case, relabel the vertex classes $(A,B,S,T)$ by $(B,A,T,S)$. Note that $G$ with this new labelling satisfies (R\ref{R1})--(R\ref{R9}) so we can find $L_1^G$ as above.}
$L_1^G$ will have the form $A*B(T)^i(S)^j(T)^kA*B$ or $A*A(T)^i(S)^j(T)^kB*B$, for some $i,j,k\geq 0$ such that $i+j+k$ is odd. Note then that $L_1^G$ will have no repeated $A$s or $B$s.

First suppose that $L_1$ is not antidirected, so $L_1$ has a consistently oriented subpath $L'$ of length two. We will find a copy of $L_1$, using (R\ref{R8})--(R\ref{R9}) to map $L'$ to a forward path of the form $ASB$ or $BTA$ or a backward path of the form $BSA$ or $ATB$. More precisely, if $L'$ is a forward path, let $L_1^G$ be a path of the form $A*ASB*B$ if $d_C(L_1, L')$ is even and a path of the form $A*BTA*B$ if $d_C(L_1, L')$ is odd. If $L'$ is backward, let $L_1^G$ be a path of the form $A*ATB*B$ if $d_C(L_1, L')$ is even and a path of the form $A*BSA*B$ if $d_C(L_1, L')$ is odd.

Suppose now that $L_1$ is antidirected. We will find a copy $L_1^G$ of $L_1$ which contains $e$. If $e\in E(B,S')$, we use (R\ref{R8}) and the definition of $S'$ to find a copy of $L_1$ of the following form.  If the initial edge of $L_1$ is a forward edge, we find $L_1^G$ of the form $A(S)^3B*B$.  If the initial edge is a backward edge, we find $L_1^G$ of the form $AB(S)^3A*B$. If $e\in E(A,T')$  we will use (R\ref{R9}) and the definition of $T'$ to find a copy of $L_1$ of the following form. If the initial edge of $L_1$ is a forward edge, we find $L_1^G$ of the form $A(T)^3B*B$.  If the initial edge is a backward edge, we find $L_1^G$ of the form $AB(T)^3A*B$.

If $L_1$ is antidirected and $e\in E(T,S')$, we will use (R\ref{R3}), (R\ref{R5}), (R\ref{R8}), (R\ref{R9}) and the definition of $S'$ to find a copy of $L_1$ containing $e$. If the initial edge of $L_1$ is a forward edge, find $L_1^G$ of the form $AB(S)^2(T)^{2h-1}A*B$, where $1\leq h\leq 2$. If the initial edge is a backward edge, find $L_1^G$ of the form $A(T)^{2h-1}(S)^2B*B$, where $1\leq h\leq 2$. Finally, we consider the case when $e\in E(S,T')$. If the initial edge of $L_1$ is a forward edge, we find $L_1^G$ of the form $AB(S)^{2h-1}(T)^2A*B$, where $1\leq h\leq 2$. If the initial edge of $L_1$ is a backward edge, we find $L_1^G$ of the form $A(T)^2(S)^{2h-1}B*B$, where $1\leq h\leq 2$.

We find a copy $L_2^G$ of $L_2$ (which is disjoint from $L_1^G$) in the same way, using the edge $f$ if $L_2$ is an antidirected path.
\end{proof}

As in the case when there were few sink vertices, we will map long paths to $G[S]$ and $G[T]$. It will require considerable work to choose these paths so that $G$ contains edges which can be used to link these paths together and so that we are able to cover the remaining vertices in $S\cup T$ using sink and source vertices in a ``balanced'' way. In many ways, the proof is similar to the proof of Lemma~\ref{lem:excover2}. In particular, we will use Proposition~\ref{prop:sinksource} to map sink and source vertices to some vertices in $S\cup T$.

\begin{lemma}\label{lem:ABST2}
Suppose $1/n \ll \eps \ll \eps_1 \ll \eps_2\ll \eta_1 \ll \tau \ll 1.$
Let $G$ be a digraph on $n$ vertices with $\delta^0(G) \geq n/2$. Suppose $A, B, S, T$ is a partition of $V(G)$ satisfying (R\ref{R*})--(R\ref{R9}). Let $C$ be an oriented cycle on $n$ vertices which is not antidirected. If $\sigma(C) \geq \eps_2n$, then $G$ has an exceptional cover $P$ such that $|V(P)\cap(A\cup B)|\leq 5\eps_2n$.
\end{lemma}

\begin{proof}
Let $d:=b-a$. Define $S'\subseteq S$ to be the set consisting of all vertices $x\in S$ with $d^\pm_{S}(x)\geq \eta_1 n/2$ and define $T'\subseteq T$ similarly. Let $S'':=S\setminus S'$ and $T'':=T\setminus T'$. Note that $|S''|, |T''| \leq \eps_1n$ by (R\ref{R8}) and (R\ref{R9}). By (R\ref{R4}), all vertices $x\in S''$ satisfy $d^-_A(x)\geq \eta_1 n/2$ or $d^+_B(x) \geq \eta_1 n/2$ and, by (R\ref{R5}), all $x\in T''$ satisfy $d^+_A(x)\geq \eta_1 n/2$ or $d^-_B(x) \geq \eta_1 n/2$. In our proof below, we will find disjoint sets $S_A, S_B \subseteq S$ and $T_A, T_B \subseteq T$ of suitable size such that
\begin{align}
&d^-_A(x)\geq \eta_1 n/2 \text{ for all } x\in S_A\;\text{ and }\;
d^+_B(x)\geq \eta_1 n/2 \text{ for all } x\in S_B;\label{eq:condition1}\\
&d^-_B(x)\geq \eta_1 n/2 \text{ for all } x\in T_B \;\text{ and }\;
d^+_A(x)\geq \eta_1 n/2 \text{ for all } x\in T_A.\label{eq:condition2}
\end{align}
Note that (R\ref{R8}) implies that all but at most $\eps_1n$ vertices from $S$ could be added to $S_A$ or $S_B$ and satisfy the conditions of \eqref{eq:condition1}. Similarly, (R\ref{R9}) implies that all but at most $\eps_1n$ vertices in $T$ are potential candidates for adding to $T_A$ or $T_B$ so as to satisfy \eqref{eq:condition2}. We will write $s_A:=|S_A|$, $s_B:=|S_B|$, $t_A:=|T_A|$ and $t_B:=|T_B|$.

Let $s^*:=s-\lceil\sqrt{\eps_1} n\rceil$ and let $\ell:=2\lceil \eps_2n\rceil -1$. If $C$ contains an antidirected subpath of length $\ell$, let $Q_2$ denote such a path. We may assume that the initial edge of $Q_2$ is a forward edge by reordering the vertices of $C$ if necessary. Otherwise, choose $Q_2$ to be any subpath of $C$ of length $\ell$ such that $Q_2$ contains at least $\eps_1^{1/3} n$ sink vertices and the second vertex of $Q_2$ is a sink. Let $Q_1$ be the subpath of $C$ of length $\ell$ such that $d_C(Q_1,Q_2)=2s^*+\ell$. Note that if $Q_1$ is antidirected then $Q_2$ must also be antidirected. Let $e_1,e_2$ be the final two edges of $Q_1$ and let $f_1,f_2$ be the initial two edges of $Q_2$ (where the edges are listed in the order they appear in $Q_1$ and $Q_2$, i.e., $(e_1e_2)\subseteq Q_1$ and $(f_1f_2)\subseteq Q_2$). Note that $f_1$ is a forward edge and $f_2$ is a backward edge.

Let $Q'$ be the subpath of $C$ of length $14$ such that $d_C(Q', Q_2)=s^*$. If $Q'$ is antidirected, let $Q$ be the subpath of $Q'$ of length $13$ whose initial edge is a forward edge. Otherwise let $Q\subseteq Q'$ be a consistently oriented path of length two. We will consider the three cases stated below.

\medskip

\noindent \textbf{Case 1: }\emph{$Q_1$ and $Q_2$ are antidirected. Moreover, $\{e_2,f_1\}$ is consistent if and only if $n$ is even.}

We will assume that the initial edge of $Q$ is a forward edge, the case when $Q$ is a backward path of length two is very similar.%
\COMMENT{Suppose that $Q$ is a backward path. Map $Q$ to a backward path $Q^G$ of the form $T'BS'$. If $n$ is even, let $e:=e_1$ and, if $n$ is odd, let $e:=e_2$. In both cases, let $f:=f_2$. The assumptions of this case imply that $e$ and $f$ are both backward edges. Define $P$, $P_T$, $P_S$, $p_T$, $p_S$, $q_T$ and $q_S$ as in the main text. We have that $p_S+p_T+q_T+q_S=d_C(e,f)-1\equiv n \mod 2$. So we can choose $S_A, S_B, T_A, T_B$ as in main text.\\
Recall that $Q_1$ is antidirected. So we can find a path $(Q_1e)^G$ isomorphic to $(Q_1e)$ which covers the vertices in $T_A$ by source vertices and the vertices in $T_B$ by sink vertices - parity is OK as $e$ is backward. We choose this path to have the form 
$$X*A(BAT_AA*A)^{t_A}(BT_BB*A)^{t_B}B*AT',$$ where $X\in \{A,B\}$. Observe that $(Q_1e)^G$ has $t_A$ repeated $A$s and $t_B$ repeated $B$s. Find a path $(fQ_2)^G$ isomorphic to $(fQ_2)$ of the form
$$S'A*A(BAS_AA*A)^{s_A}(BS_BB*A)^{s_B}B*B$$
which covers all vertices in $S_A$ by sink vertices and all vertices in $S_B$ by source vertices. $(fQ_2)^G$ has $s_A$ repeated $A$s and $s_B$ repeated $B$s. The rest of the proof is identical to the case when $Q$ is a forward path.}
We will find a copy $Q^G$ of $Q$ which is a $T'S'$-path. If $Q$ is a forward path of length two, map $Q$ to a forward path $Q^G$ of the form $T'AS'$. If $Q$ is antidirected, we find a copy $Q^G$ of $Q$ as follows. Let $Q''$ be the subpath of $Q$ of length eight such that $d_C(Q, Q'')=3$. Recall that a path in $G$ is useful if it has no repeated $A$s or $B$s and uses an odd number of vertices from $S\cup T$. Using Proposition~\ref{prop:ABST2links}, we find a copy $(Q'')^G$ of $Q''$ in $G$ which is a useful $AB$-path. We find $Q^G$ which starts with a path of the form $T'ABA$, uses $(Q'')^G$ and then ends with a path of the form $BAS'$. Let $q_S$ and  $q_T$ be the numbers of interior vertices of $Q^G$ in $S$ and $T$, respectively.

If $n$ is even, let $e:=e_2$ and, if $n$ is odd, let $e:=e_1$. In both cases, let $f:=f_1$. The assumptions of this case imply that $e$ and $f$ are both forward edges. Let $P:=(Q_1CQ_2)$ and let $P_T$ and $P_S$ be subpaths of $C$ which are internally disjoint from $e, f$ and $Q$ and are such that $(eCf)= (eP_TQP_Sf)$. Our plan is to find a copy of $P_T$ in $G[T]$ and a copy of $P_S$ in $G[S]$. Let $p_T:=|P_T|$ and $p_S:=|P_S|$. If $Q$ is a consistently oriented path we have that $q_S,q_T=0$ and $p_S+p_T=d_C(e,f)-1$. If $Q$ is antidirected, then $q_S+q_T$ is odd and $p_S+p_T=d_C(e,f)-12$. So in both cases we observe that
\begin{equation}\label{eqn:parity1}
p_S+p_T+q_S+q_T\equiv d_C(e,f)-1\equiv n\mod 2.
\end{equation}

Choose $S_A, S_B, T_A, T_B$ to satisfy \eqref{eq:condition1} and \eqref{eq:condition2} so that $S''\setminus V(Q^G)\subseteq S_A\cup S_B$, $T''\setminus V(Q^G)\subseteq T_A\cup T_B$, $s=s_A+s_B+p_S+q_S$, $t=t_A+t_B+p_T+q_T$ and $s_A+t_A+d=s_B+t_B$. To see that this can be done, first note that the choice of $s^*$ implies that $s-p_S-q_S \geq \sqrt{\eps_1}n/2>|S''|+d$ and $t-p_T-q_T\geq \sqrt{\eps_1}n/2 > |T''|+d$. Let $r:=s+t-(p_S+p_T+q_S+q_T)$. So $r$ is the number of vertices in $S\cup T$ which will not be covered by the copies of $P_T$, $P_S$ or $Q$. Then \eqref{eqn:parity1} implies that 
$$r\equiv s+t-n \equiv d \mod 2.$$ Thus we can choose the required subsets $S_A, S_B, T_A, T_B$ so that $s_A+t_A=(r-d)/2$ and $s_B+t_B=(r+d)/2$. Note that (R\ref{R2}) and the choice of $s^*$ also imply that $s_A+s_B, t_A+t_B \leq 2\sqrt{\eps_1}n$.

Recall that $Q_1$ is antidirected. So we can find a path $(Q_1e)^G$ isomorphic to $(Q_1e)$ which covers the vertices in $T_A$ by source vertices and the vertices in $T_B$ by sink vertices. We choose this path to have the form%
\COMMENT{Since $e$ is forward and the initial vertex of $e$ is mapped to $B$, the parity is OK, that is, we can map source vertices to $T_A$ and sink vertices to $T_B$.}
$$X*A(BAT_AA*A)^{t_A}(BT_BB*A)^{t_B}B*BT',$$ where $X\in \{A,B\}$. Observe that $(Q_1e)^G$ has $t_A$ repeated $A$s and $t_B$ repeated $B$s. Find a path $Q_2^G$ isomorphic to $Q_2$ of the form
$$S'B*A(BAS_AA*A)^{s_A}(BS_BB*A)^{s_B}B*B$$
which covers all vertices in $S_A$ by sink vertices and all vertices in $S_B$ by source vertices. $Q_2^G$ has $s_A$ repeated $A$s and $s_B$ repeated $B$s. So far, we have been working under the assumption that $Q$ starts with a forward edge. If $Q$ is a backward path, the main difference is that we let $e:=e_1$ if $n$ is even and let $e:=e_2$ if $n$ is odd. We let $f:=f_2$ so that $e$ and $f$ are both backward edges and we map $Q$ to a backward path $Q^G$ of the form $T'BS'$. Then \eqref{eqn:parity1} holds and we can proceed similarly as in the case when $Q$ is a forward path.

We find copies of $P_T$ in $G[T']$ and $P_S$ in $G[S']$ as follows. Greedily embed the first $\sqrt{\eps_1}n$ vertices of $P_T$ to cover all uncovered vertices $x\in T'$ with $d^+_T(x)\leq t-\eps^{1/3}n$ or $d^-_T(x)\leq t-\eps^{1/3}n$. Note that, by (R\ref{R9}), there are at most $\eps_1n$ such vertices. Write $P_T'\subseteq P_T$ for the subpath still to be embedded and let $t_1$ and $t_2$ be the images of its endvertices in $T$. Let $T^*$ denote the sets of so far uncovered vertices in $T$ together with $t_1$ and $t_2$ and define $G_T:=G[T^*]$. We have that $\delta^0(G_T)\geq t-\eps^{1/3}n-3\sqrt{\eps_1}n\geq 7|G_T|/8$, using (R\ref{R1}), and so we can apply Proposition~\ref{prop:completepath}(i) to find a copy of $P_T'$ in $G_T$ with the desired endpoints. In the same way, we find a copy of $P_S$ in $G[S']$. Together with $Q^G$, $(Q_1e)^G$ and $Q_2^G$, this gives a copy $P^G$ of $P$ in $G$ such that $|V(P^G)\cap(A\cup B)|\leq 5\eps_2n$.

The path $P^G$ satisfies (EC\ref{EC1}) and we may assume that (EC\ref{EC2}) holds, by extending the path by one or two vertices, if necessary, so that both of its endvertices lie in $A$. Let us now verify (EC\ref{EC3}). All repeated $A$s and $B$s in $P^G$ are repeated $A$s and $B$s in the paths $(Q_1e)^G$ and $Q_2^G$. So in total, $P^G$ has $s_A+t_A$ repeated $A$s and $s_B+t_B$ repeated $B$s. Then \eqref{eqn:repeats} gives that $P^G$ satisfies 
$$|B\setminus V(P^G)|-|A\setminus V(P^G)|=d-(s_B+t_B)+(s_A+t_A)+1=1.$$
 So (EC\ref{EC3}) is satisfied and $P^G$ is an exceptional cover.

\medskip

\noindent \textbf{Case 2: }\emph{There exists $e\in \{e_1,e_2\}$ and $f\in \{f_1,f_2\}$  such that $\{e,f\}$ is consistent and $n-d_C(e,f)$ is even.}

Let $v$ be the final vertex of $f$. Recall the definitions of a useful tripartition and a link from Section~\ref{sec:AB}. Consider a useful tripartition $P_1,P_2,P_3$ of $(vQ_2)$ and let $\cQ_1, \cQ_2, \cQ_3$ be sink/source/sink sets. Let $L\subseteq P_2$ be a link of length eight such that $d_C(v, L)$ is even. If $Q$ is a consistently oriented path, use Proposition~\ref{prop:ABST2links} to find a copy $L^G$ of $L$ which is a useful $BA$-path if $e$ is forward and a useful $AB$-path if $e$ is backward.  Map $Q$ to a path $Q^G$ of the form $T'AS'$ if $Q$ is a forward path and $T'BS'$ if $Q$ is a backward path. If $Q$ is antidirected, let $Q''$ be the subpath of $Q$ of length eight such that $d_C(Q, Q'')=3$. Using Proposition~\ref{prop:ABST2links}, we find disjoint copies $(Q'')^G$ of $Q''$ and $L^G$ of $L$ in $G$ such that $(Q'')^G$ is a useful $AB$-path and $L^G$ is as described above. We find $Q^G$ which starts with a path of the form $T'ABA$, uses $(Q'')^G$ and then ends with a path of the form $BAS'$. Let $q_S$ be the number of interior vertices of $Q^G$ and $L^G$ in $S$ and let $q_T$ be the number of interior vertices of $Q^G$ and $L^G$ in $T$. Note that in all cases, $Q^G$ is a $T'S'$-path with no repeated $A$s or $B$s.

Let $P:=(eCQ_2)$ and let $P_0:=(eCf)$. Define subpaths $P_T$ and $P_S$ of $C$ which are internally disjoint from $Q,e,f$ and are such that $P_0= (eP_TQP_Sf)$. Let $p_T:=|P_T|$ and $p_S:=|P_S|$.  Our aim will be to find a copy $P_0^G$ of $P_0$ which uses $Q^G$ and maps $P_T$ to $G[T]$ and $P_S$ to $G[S]$. $P_0^G$ will have the form $F$ given in Table~\ref{table2}. We fix edges $e^G$ and $f^G$ for $e$ and $f$. If $e$ is a forward edge, then choose $e^G$ to be a $BT'$-edge and $f^G$ to be an $S'B$-edge. If $e$ is a backward edge, let $e^G$ be a $T'A$-edge and $f^G$ be an $AS'$-edge. We also define a constant $d'$ in Table~\ref{table2} which will be used to ensure that the final assignment is balanced.
\begin{table}[h]
	\begin{tabular}{|c|c|c|c|c|}
	\hline	
	Initial edge of $Q$& forward  & forward & backward & backward\\
	$e$& forward & backward & forward & backward\\
	\hline
	\rule{0pt}{11pt}$F$ & $BT^{p_T}\cA S^{p_S}B$& $AT^{p_T}\cA S^{p_S}A$ & $BT^{p_T}BS^{p_S}B$ & $AT^{p_T}BS^{p_S}A$\\
	\hline
	\rule{0pt}{11pt}$d'$ & $d$& $d+2$ & $d-2$ & $d$\\
	\hline
	\end{tabular}
	\caption{Proof of Lemma~\ref{lem:ABST2}, Cases~2 and 3: $P_0^G$ has form $F$, where $\cA$ denotes an $A$-path with no repeated $A$s or $B$s.}\label{table2}
\end{table}
So, if $r_A$ and $r_B$ are the numbers of repeated $A$s and $B$s in $P_0^G$ respectively, we will have $r_A-r_B=d'-d$.

Note that%
\COMMENT{If $Q$ is a consistently oriented path, $p_T+p_S=d_C(e,f)-1$ and $q_T+q_S$ is the number of interior vertices in $L^G$ in $S\cup T$ which is odd. So $p_T+p_S+q_T+q_S \equiv d_C(e,f) \mod 2$. If $Q$ is antidirected, $p_T+p_S=d_C(e,f)-12$. Note that in this case $Q^G$ has an odd number of interior vertices in $S\cup T$ and $q_T+q_S$ is the number of interior vertices in $Q^G$ and $L^G$ in $S\cup T$ which is even. So $p_T+p_S+q_T+q_S \equiv d_C(e,f) \mod 2$.} 
\begin{equation}\label{eqn:parity2}
p_T+p_S+q_T+q_S \equiv d_C(e,f) \equiv n\mod 2.
\end{equation}
The number of vertices in $S\cup T$ which will not be covered by  $P_0^G$ or $L^G$ is equal to $r:=s+t-(p_T+p_S+q_T+q_S)$ and \eqref{eqn:parity2} implies that
$$r\equiv s+t-n\equiv d \equiv d' \mod 2.$$
Also note that the choice of $s^*$ implies that $s-p_S-q_S \geq \sqrt{\eps_1}n/2>|S''|+d'$ and $t-p_T-q_T\geq \sqrt{\eps_1}n/2 > |T''|+d'$.
 Thus we can choose sets $S_A, S_B, T_A, T_B$ satisfying \eqref{eq:condition1} and \eqref{eq:condition2} so that $S''\setminus V(Q^G \cup L^G)\subseteq S_A\cup S_B$, $T''\setminus V(Q^G\cup L^G)\subseteq T_A\cup T_B$, $s=s_A+s_B+p_S+q_S$, $t=t_A+t_B+p_T+q_T$ and $s_A+t_A+d'=s_B+t_B$.%
\COMMENT{Let $s_A+t_A=(r-d')/2$ and $s_B+t_B=(r+d')/2$.} 
(R\ref{R2}) and the choice of $s^*$ imply that $s_A+s_B, t_A+t_B \leq 2\sqrt{\eps_1}n$.  
Recall that $v$ denotes the final vertex of $f$ and let $v^G$ be the image of $v$ in $G$. If $v^G\in A$ (i.e., if $e$ is backward), let $v':=v$ and $(v')^G:=v^G$. If $v^G\in B$, let $v'$ denote the successor of $v$ on $C$. If $vv'\in E(C)$, map $v'$ to an outneighbour of $v^G$ in $A$ and, if $v'v\in E(C)$, map $v'$ to an inneighbour of $v^G$ in $A$. Let $(v')^G$ be the image of $v'$. Then we can apply Proposition~\ref{prop:sinksource}, with $2\sqrt{\eps_1}, \eta_1/2, \tau/2, (v')^G$ playing the roles of $\eps, \eta, \tau, a_1$, to find a copy $(v'Q_2)^G$ of $(v'Q_2)$ which starts at $(v')^G$, covers $S_A, S_B, T_A, T_B$ and contains $L^G$. Note that we make use of \eqref{eq:condition1} and \eqref{eq:condition2} here. We obtain a copy $(vQ_2)^G$ of $(vQ_2)$ (by combining $v^G(v')^G$ with $(v'Q_2)^G$ if $v'\neq v$) which has $s_A+t_A$ repeated $A$s and $s_B+t_B$ repeated $B$s.

We find copies of $P_T$ in $G[T]$ and $P_S$ in $G[S]$ as in Case~1. Combining these paths with $(vQ_2)^G$, $e^G$, $Q^G$ and $f^G$, we obtain a copy $P^G$ of $P$ in $G$ such that $|V(P^G)\cap(A\cup B)|\leq 3\eps_2 n$. The path $P^G$ satisfies (EC\ref{EC1}) and we may assume that (EC\ref{EC2}) holds, by extending the path if necessary to have both endvertices in $A$. All repeated $A$s and $B$s in $P^G$ occur as repeated $A$s and $B$s in the paths $P_0^G$ and $(vQ_2)^G$ so we can use \eqref{eqn:repeats} to see that 
$$|B\setminus V(P^G)|-|A\setminus V(P^G)|= d-(s_B+t_B)+(d'-d)+(s_A+t_A)+1=1.$$
Therefore, (EC\ref{EC3}) is satisfied and $P^G$ is an exceptional cover.

\medskip

\noindent \textbf{Case 3: }\emph{The assumptions of Cases~$1$ and $2$ do not hold.}

Recall that $f_1$ is a forward edge and $f_2$ is a backward edge. Since Case~2 does not hold, this implies that $e_2$ is a forward edge if $n$ is even (otherwise $e:=e_2$ and $f:=f_2$ would satisfy the conditions of Case~2) and $e_2$ is a backward edge if $n$ is odd (otherwise $e:=e_2$ and $f:=f_1$ would satisfy the conditions of Case~2). In particular, since Case~1 does not hold, this in turn implies that $Q_1$ is not antidirected. 
We claim that $Q_1\setminus \{e_2\}$ is not antidirected. Suppose not. Then it must be the case that $\{e_1,e_2\}$ is consistent. If $e_1$ and $e_2$ are forward edges (and so $n$ is even), then $e:=e_1$ and $f:=f_1$ satisfy the conditions of Case~2. If  $e_1$ and $e_2$ are both backward edges (and so $n$ is odd), then $e:=e_1$ and $f:=f_2$ satisfy the conditions of Case~2. Therefore, $Q_1\setminus \{e_2\}$ is not antidirected and must contain a consistently oriented path $Q_1'$ of length two.

Let $e:=e_2$. If $n$ is even, let $f:=f_1$ and, if $n$ is odd, let $f:=f_2$. In both cases, we have that $\{e,f\}$ is consistent.  Let $P:=(Q_1'CQ_2)$ and $P_0:=(ePf)$. Let $P_T$ and $P_S$ be subpaths of $C$ defined such that $P_0=(eP_TQP_Sf)$. Set $p_T:=|P_T|$ and $p_S:=|P_S|$. Our aim is to find a copy $P_0^G$ which is of the form given in Table~\ref{table2}. We also define a constant $d'$ as in Table~\ref{table2}. So if $r_A$ and $r_B$ are the numbers of repeated $A$s and $B$s in $P_0^G$ respectively, then again $r_A-r_B=d'-d$. 

Let $v$ be the final vertex of $f$. Consider a tripartition $P_1, P_2, P_3$ of $(vQ_2)$ and a link $L\subseteq P_2$ of length eight such that $d_C(v,L)$ is even. Proceed exactly as in Case~2 to find copies $Q^G$ and $L^G$ of $Q$ and $L$.
Use (R\ref{R3}), (R\ref{R8}) and (R\ref{R9}) to fix a copy $(Q_1'Ce)^G$ of $(Q_1'Ce)$ which is disjoint from $Q^G$ and $L^G$ and is of the form given in Table~\ref{table3}.
\begin{table}[h]
	\begin{tabular}{|p{3.2cm}|c|c|c|c|}
	\hline	
	\centering{$Q_1'$}& forward  & forward & backward & backward\\
	\centering{$d_C(Q_1',e)$}& odd & even & odd & even\\
	\hline
	\rule{0pt}{11pt} \centering{Form of $(Q_1'Ce)^G$ if $e$ is forward}& $BTA*BT'$ & $ASB*BT'$ & $BSA*BT'$ & $ATB*BT'$\\
	\hline
		\rule{0pt}{11pt} \centering{Form of $(Q_1'Ce)^G$ if $e$ is backward}& $ASB*AT'$ & $BTA*AT'$ & $ATB*AT'$ & $BSA*AT'$\\
	\hline
	\end{tabular}
	\caption{Form of $(Q_1'Ce)^G$ in Case 3.}\label{table3}
\end{table}
Note that the interior of $(Q_1'Ce)^G$ uses exactly one vertex from $S\cup T$ and $(Q_1'Ce)^G$ has no repeated $A$s or $B$s. Write $(Q_1')^G$ for the image of $Q_1'$. We also fix an edge $f^G$ for the image of $f$ which is disjoint from $Q^G$, $L^G$ and $(Q_1'Ce)^G$ and is an $S'B$-edge if $e$ is forward and an $AS'$-edge if $e$ is backward. Let $q_S$ be the number of interior vertices of $Q^G$, $L^G$ and   $(Q_1')^G$ in $S$ and let $q_T$ be the number of interior vertices of $Q^G$, $L^G$ and  $(Q_1')^G$ in $T$. 

Note that $p_S+p_T+q_S+q_T\equiv d_C(e,f)-1 \equiv n \mod 2$.%
\COMMENT{Note that $(Q_1')^G$ has exactly one interior vertex in $S\cup T$. If $Q$ is consistently oriented, $p_T+p_S=d_C(e,f)-1$ and $q_T+q_S$ is even. If $Q$ is antidirected oriented, $p_T+p_S=d_C(e,f)-12$ and $q_T+q_S$ is odd.} 
Using the same reasoning as in Case~2, we find sets $S_A, S_B, T_A, T_B$ satisfying \eqref{eq:condition1} and \eqref{eq:condition2} such that $S''\setminus V(Q^G \cup L^G\cup (Q_1')^G)\subseteq S_A\cup S_B$, $T''\setminus V(Q^G\cup L^G\cup (Q_1')^G)\subseteq T_A\cup T_B$, $s=s_A+s_B+p_S+q_S$, $t=t_A+t_B+p_T+q_T$ and $s_A+t_A+d'=s_B+t_B$.
(R\ref{R2}) and the choice of $s^*$ imply that $s_A, t_A, s_B, t_B \leq 2\sqrt{\eps_1}n$.  
Recall that $v$ denotes the final vertex of $f$. Similarly as in Case~2, we now use Proposition~\ref{prop:sinksource} to find a copy $(vQ_2)^G$ of $(vQ_2)$ which covers $S_A, S_B, T_A, T_B$, contains $L^G$ and has $s_A+t_A$ repeated $A$s and $s_B+t_B$ repeated $B$s.%
\COMMENT{Let $v^G$ be the image of $v$. If $v^G\in A$, let $v':=v$ and $(v')^G:=v^G$. If $v^G\in B$, let $v'$ denote the successor of $v$ on $C$. Map $v'$ to a suitable neighbour $(v')^G$ of $v^G$ in $A$. Then we can apply Proposition~\ref{prop:sinksource}, with $2\sqrt{\eps_1}, \eta_1/2, \tau/2, (v')^G$ playing the roles of $\eps, \eta, \tau, a_1$, to find a copy $(v'Q_2)^G$ of $(v'Q_2)$ which starts at $(v')^G$, covers $S_A, S_B, T_A, T_B$ and contains $L^G$. We obtain a copy $(vQ_2)^G$ of $(vQ_2)$ which has $s_A+t_A$ repeated $A$s and $s_B+t_B$ repeated $B$s.}

We find copies of $P_T$ in $G[T]$ and $P_S$ in $G[S]$ as in Case~1. Together with $(Q_1'Ce)^G$, $Q^G$, $f^G$ and $(vQ_2)^G$, these paths give a copy $P^G$ of $P$ in $G$ such that $|V(P^G)\cap(A\cup B)|\leq 5\eps_2 n$. The path $P^G$ satisfies (EC\ref{EC1}) and we may assume that (EC\ref{EC2}) holds, by extending the path so that both endvertices lie in $A$ if necessary. All repeated $A$s and $B$s in $P^G$ occur as repeated $A$s and $B$s in the paths $P_0^G$ and $(vQ_2)^G$, so we can use \eqref{eqn:repeats} to see that 
$$|B\setminus V(P^G)|-|A\setminus V(P^G)|= d-(s_B-t_B)-(d-d')+(s_A+t_A)+1=1.$$
So (EC\ref{EC3}) is satisfied and $P^G$ is an exceptional cover. 
\end{proof}

\subsection{Finding a copy of $C$}

As we did in the $AB$-extremal case, we will now use an exceptional cover to find a copy of $C$ in $G$.

\begin{proofof}\textbf{Lemma~\ref{lem:ABST}.}
Apply Lemma~\ref{lem:ABST1} or Lemma~\ref{lem:ABST2} to find an exceptional cover $P$ of $G$ which uses at most $2\eta_1^2n$ vertices from $A\cup B$. Let $P'$ be the path of length  $\sqrt{\eps_1}n$ following $P$ on $C$. Extend $P$ by a path isomorphic to $P'$, using this path to cover all $x\in A$ such that $d^+_B(x)\leq b-\eps^{1/3}n$ or $d^-_B(x)\leq b-\eps^{1/3}n$ and all $x\in B$ such that $d^+_A(x)\leq a-\eps^{1/3}n$ or $d^-_A(x)\leq a-\eps^{1/3}n$, using only edges in $E(A,B)\cup E(B,A)$. Let $P^*$ denote the resulting extended path.

We may assume that both endvertices $a_1, a_2$ of $P^*$ are in $A$ and also that $d^\pm_B(a_i) \geq b-\eps^{1/3} n$ (by extending the path if necessary). Let $A^*, B^*$ denote those vertices in $A$ and $B$ which have not already been covered by $P^*$ together with $a_1$ and $a_2$ and let $G^*:=G[A^*, B^*]$. We have that $|A^*|=|B^*|+1$ and
$\delta^0(G^*)\geq a-3\eta_1^2n \geq (7|B^*|+2)/8.$
Then $G^*$ has a Hamilton path of any orientation with the desired endpoints by Proposition~\ref{prop:completepath}(ii). Together with $P^*$, this gives a copy of $C$ in $G$.
\end{proofof}

\section*{Acknowledgements}
We are grateful to the referees for a careful reading of this paper.

\medskip

{\footnotesize \obeylines \parindent=0pt

Louis DeBiasio
Department of Mathematics
Miami University
Oxford
OH 45056
USA
}
\begin{flushleft}
{\it{E-mail address}:
\tt{debiasld@miamioh.edu}}
\end{flushleft}

{\footnotesize \obeylines \parindent=0pt

Daniela K\"{u}hn, Deryk Osthus, Amelia Taylor 
School of Mathematics
University of Birmingham
Edgbaston
Birmingham
B15 2TT
UK
}
\begin{flushleft}
{\it{E-mail addresses}:
\tt{\{d.kuhn, d.osthus\}@bham.ac.uk}, a.m.taylor@pgr.bham.ac.uk}
\end{flushleft}

{\footnotesize \obeylines \parindent=0pt
Theodore Molla
Department of Mathematics
University of Illinois at Urbana-Champaign
Urbana
IL 61801
USA
}
\begin{flushleft}
{\it{E-mail address}:
\tt{molla@illinois.edu}}
\end{flushleft}

\end{document}